\documentclass[a4paper, 11pt]{amsart}
\usepackage{amsmath, amssymb, amsthm, mathrsfs}
\usepackage[mathscr]{eucal}
\usepackage{graphicx}

\theoremstyle{plain} 
 \newtheorem{thm}{Theorem}[section]
 \newtheorem{lem}[thm]{Lemma}

 \newtheorem{claim}[thm]{Claim}
\theoremstyle{definition}

\theoremstyle{remark}
  
  \newtheorem{ex}[thm]{Example}

\newcommand{\R}{\mathbb{R}}
\newcommand{\cal}{\mathcal}
\newcommand{\N}{\mathbb{N}}
\newcommand{\Z}{\mathbb{Z}}
\newcommand{\calr}{\mathcal{R}}

\newcommand{\calb}{\mathcal{B}}
\newcommand{\calg}{\mathcal{G}}
\newcommand{\calh}{\mathscr{H}}
\newcommand{\calk}{\mathscr{K}}
\newcommand{\cala}{\mathscr{A}}
\newcommand{\calbb}{\mathscr{B}}

\renewcommand{\c}{\curvearrowright}

  \makeatletter
  \@addtoreset{equation}{section}
  \makeatother

\begin{document}

\title[Stable actions of central extensions and relative property (T)]{Stable actions of central extensions \\and relative property (T)}
\author{Yoshikata Kida}
\address{Department of Mathematics, Kyoto University, 606-8502 Kyoto, Japan}
\email{kida@math.kyoto-u.ac.jp}
\date{December 18, 2013}
\subjclass[2010]{37A20}
\keywords{Discrete measured equivalence relations, discrete measured groupoids, stability, central extensions, relative property (T)}

\begin{abstract}
Let us say that a discrete countable group is stable if it has an ergodic, free, probability-measure-preserving and stable action.
Let $G$ be a discrete countable group with a central subgroup $C$.
We present a sufficient condition and a necessary condition for $G$ to be stable.
We show that if the pair $(G, C)$ does not have property (T), then $G$ is stable.
We also show that if the pair $(G, C)$ has property (T) and $G$ is stable, then the quotient group $G/C$ is stable.
\end{abstract}

\maketitle


\section{Introduction}

We mean by a {\it p.m.p.}\ action of a discrete countable group $G$ a measure-preserving action of $G$ on a standard probability space, where ``p.m.p." stands for ``probability-measure-preserving".
An ergodic, free and p.m.p.\ action is called {\it stable} if the associated equivalence relation is isomorphic to its direct product with the ergodic hyperfinite equivalence relation of type ${\rm II}_1$.
Let us say that a discrete countable group $G$ is {\it stable} if $G$ has an ergodic, free, p.m.p.\ and stable action.
A fundamental question of our interests is which group is stable.
A necessary condition for stability of a group is obtained by Jones-Schmidt, who show that any stable group is inner amenable (\cite[Proposition 4.1]{js}).
Recall that a discrete countable group $G$ is called {\it inner amenable} if there exists a sequence of non-negative $\ell^1$-functions on $G$ with unit norm, $(f_n)_{n=1}^\infty$, such that for any $g\in G$, we have
\[\lim_{n\to \infty}f_n(g)=0\quad \textrm{and}\quad \lim_{n\to \infty}\sum_{h\in G}|f_n(g^{-1}hg)-f_n(h)|=0.\]
This property was introduced by Effros \cite{effros} in connection with property Gamma of a group factor (see \cite{bh} for a survey, and \cite{csu}, \cite{choda} and \cite{vaes} for related works).
Jones-Schmidt asked whether inner amenability is sufficient for stability (\cite[Problem 4.2]{js}).
A counterexample to this question was pointed out by the author \cite{kida-inn}.
In fact, any group having property (T) and having the infinite center is inner amenable, but not stable (see \cite[Example 1.7.13]{bhv} for such groups).
The Baumslag-Solitar groups are an interesting example of inner amenable groups (\cite{stalder}), and are shown to be stable (\cite{kida-stab}).
We refer the reader to \cite[Section 9]{kec} for other results related to stability and inner amenability.

This paper focuses on groups with infinite center.
These groups are a typical example of inner amenable groups.
The following theorem gives a sufficient condition and a necessary condition for such groups to be stable.

\begin{thm}\label{thm-stab}
Let $1\to C\to G\to \Gamma \to 1$ be an exact sequence of discrete countable groups such that $C$ is central in $G$.
Then the following assertions hold:
\begin{enumerate}
\item If the pair $(G, C)$ does not have property (T), then $G$ is stable.
\item If the pair $(G, C)$ has property (T) and $G$ is stable, then $\Gamma$ is stable.
\end{enumerate}
\end{thm}

These two assertions indicate us that relative property (T) of a central subgroup plays an essential role in characterizing stability of groups with infinite center. 
We have simple applications of Theorem \ref{thm-stab} in the following:

\begin{ex}
Theorem \ref{thm-stab} (i) implies that any discrete countable group with the Haagerup property and having the infinite center is stable.
Examples of such groups are found in \cite[Sections 6.2 and 7.3.3]{ccjjv}.
\end{ex}

\begin{ex}
Let $H$ be a discrete countable group having property (T) and having an infinite central subgroup $C$.
Set $G=H\ast_C H$, the amalgamated free product with respect to the inclusion of $C$ into two copies of $H$.
We have the exact sequence
\[1\to C\to G\to (H/C)\ast (H/C)\to 1.\]
The free product $(H/C)\ast (H/C)$ is not stable, by computation of cost (\cite{gab-cost}) or $\ell^2$-Betti numbers (\cite{gab-l2}), or invariance under measure equivalence of vanishing of a certain bounded cohomology group (\cite[Section 7]{ms}).
The pair $(G, C)$ has property (T), and $G$ is not stable by Theorem \ref{thm-stab} (ii).
\end{ex}

Jones-Schmidt \cite{js} obtain a useful criterion for stability of an ergodic, free and p.m.p.\ action.
It asserts that if there exists a certain sequence in the full group of the associated equivalence relation, called a non-trivial asymptotically central (a.c.) sequence, then the action is stable.
In the proof of Theorem \ref{thm-stab}, we construct a p.m.p.\ action which admits a non-trivial a.c.\ sequence in the full group of the associated groupoid, but is not necessarily ergodic or free.
The following theorem is technically useful, asserting that existence of such a p.m.p.\ action is sufficient for stability of a group.
We refer to Subsection \ref{subsec-stab} for a definition of an a.c.\ sequence for a p.m.p.\ action and its triviality.

\begin{thm}\label{thm-ac}
Let $G$ be a discrete countable group.
If there exists a p.m.p.\ action of $G$ admitting a non-trivial asymptotically central sequence in the full group of the associated groupoid, then $G$ is stable.
\end{thm}

Theorem \ref{thm-stab} (i) is strongly inspired by Connes-Weiss' theorem \cite{cw} stating that any group without property (T) has an ergodic p.m.p.\ action admitting a non-trivial asymptotically invariant (a.i.) sequence.
An a.i.\ sequence for an ergodic p.m.p.\ action was introduced by Schmidt \cite{sch}, who shows that any ergodic p.m.p.\ action of a property (T) group admits no non-trivial a.i.\ sequence (see Subsection \ref{subsec-stab} for a definition of an a.i.\ sequence).
These two assertions characterize property (T) groups in terms of their ergodic p.m.p.\ actions.
We refer to \cite{sch-ame} for further results related to this characterization. 
It is notable that from denying property (T), Connes-Weiss construct an ergodic p.m.p.\ action enjoying a certain asymptotic property.
The desired action is obtained through Gaussian actions associated with orthogonal representations of a group.
The proof of Theorem \ref{thm-stab} (i) heavily relies on their construction.

\medskip

The paper is organized as follows:
In Section \ref{sec-pre}, we review property (T) of the pair of a group and its subgroup.
We also review construction of the Gaussian action associated with a unitary representation.
In Section \ref{sec-js}, we prove Theorem \ref{thm-ac}.
In Sections \ref{sec-ce} and \ref{sec-cq}, we prove assertions (i) and (ii) in Theorem \ref{thm-stab}, respectively.

\medskip

Let us summarize notation and terminology employed throughout the paper.
We mean by a discrete group a discrete and countable group.
We mean by a {\it standard probability space} a standard Borel space equipped with a probability measure.
All relations among measurable sets and maps that appear in the paper are understood to hold up to sets of measure zero, unless otherwise stated.
Let $\N$ denote the set of positive integers.
For two subsets $A$ and $B$ of a set, the symbol $A\bigtriangleup B$ stands for the symmetric difference between $A$ and $B$, i.e., the set $(A\setminus B)\cup (B\setminus A)$.



\section{Preliminaries}\label{sec-pre}

\subsection{Relative property (T)}

We review this property, and give its consequence toward certain linear isometric representations on $L^p$ spaces.
We recommend the reader to consult \cite{bfgm} and \cite{bhv} for details.

\subsubsection{Terminology}

Let $G$ be a discrete group.
Let $\cala$ be a real or complex Banach space.
We denote by $O(\cala)$ the group of invertible linear isometries on $\cala$.
We mean by a linear isometric (or orthogonal) representation of $G$ on $\cala$ a homomorphism $\pi \colon G\to O(\cala)$.
We often write $(\pi, \cala)$ for such a representation.
A linear isometric representation of $G$ on a complex Hilbert space is called a unitary representation of $G$.
We denote by $1_G$ the trivial representation of $G$ if the scalar field can be understood from the context.

Let $(\pi, \cala)$ be a linear isometric representation of $G$.
We say that a sequence $(v_n)_n$ in $\cala$ is {\it asymptotically $G$-invariant} under $\pi$ if for any $g\in G$, we have $\Vert \pi(g)v_n-v_n\Vert \to 0$ as $n\to \infty$.
We say that the representation $\pi$ {\it almost has $G$-invariant vectors} if there exists a sequence of unit vectors in $\cala$ which is asymptotically $G$-invariant under $\pi$.

Let $C$ be a subgroup of $G$.
We say that the pair $(G, C)$ has {\it property (T)} (or $C$ has {\it relative property (T)} in $G$) if any unitary representation of $G$ almost having $G$-invariant vectors has a non-zero $C$-invariant vector.
It is known that the pair $(G, C)$ has property (T) if and only if for any net $(\pi_i)_i$ in the unitary dual of $G$ converging to $1_G$, the representation $\pi_i$ eventually has a non-zero $C$-invariant vector (\cite[Remark 1.4.4 (vi)]{bhv}).
We say that $G$ has {\it property (T)} if the pair $(G, G)$ has property (T).

\subsubsection{Certain representations on $L^p$ spaces}\label{subsubsec-lp}

Results in this subsubsection will be used in Section \ref{sec-cq} for the proof of Theorem \ref{thm-stab} (ii).

Let $G$ be a discrete group, $C$ a central subgroup of $G$, and $G\c (X, \mu)$ a p.m.p.\ action.
Let $\theta \colon (X, \mu)\to (Z, \xi)$ be the ergodic decomposition for the action $C\c (X, \mu)$.
We have the disintegration $\mu =\int_Z\mu_z\, d\xi(z)$ of $\mu$ with respect to $\theta$.

Let $p$ be a real number with $p\geq 1$.
We set $\cala_p=L^p(G\times X, \R)$, the Banach space of real-valued $p$-integrable functions on $G\times X$, where $G\times X$ is equipped with the product measure of the counting measure on $G$ and $\mu$.
Let $\pi_p$ be the linear isometric representation of $G$ on $\cala_p$ defined by
\[(\pi_p(g)f)(h, x)=f(g^{-1}hg, g^{-1}x)\quad \textrm{for}\ g, h\in G,\ f\in \cala_p \ \textrm{and}\ x\in X.\]  
Let $\cala_p^C$ denote the subspace of $\cala_p$ of $C$-invariant vectors under $\pi_p$.
Since $C$ is central in $G$, the space $\cala_p^C$ consists of all vectors $f$ in $\cala_p$ such that for any $g\in G$, the function $f(g, \cdot)$ on $X$ is $C$-invariant. 
We define a linear map $P\colon \cala_p \to \cala_p$ by
\[P(f)(g, x)=\mu_{\theta(x)}(f(g, \cdot))\quad \textrm{for}\ f\in \cala_p,\ g\in G\ \textrm{and}\ x\in X.\]
The map $P$ is the projection onto $\cala_p^C$.
We set $\calk_p=\ker P$.

\begin{lem}\label{lem-p}
In the above notation, we suppose that the pair $(G, C)$ has property (T).
Let $(v_n)_{n=1}^\infty$ be a sequence of unit vectors in $\cala_p$ which is asymptotically $G$-invariant under $\pi_p$.
Then we have $\Vert v_n-Pv_n\Vert \to 0$ as $n\to \infty$.
\end{lem}

\begin{proof}
We follow \cite[Section 4.a]{bfgm}.
We define the Mazur map $M_{p, 2}\colon \cala_p\to \cala_2$ by $M_{p, 2}(f)={\rm sign}(f)|f|^{p/2}$ for $f\in \cala_p$, where ${\rm sign}(f)$ is the function on $G\times X$ taking values $1$, $0$ and $-1$ on the sets $\{ f>0\}$, $\{ f=0 \}$ and $\{ f<0\}$, respectively.
The map $M_{p, 2}$ is not linear, but induces a uniformly continuous homeomorphism from the unit sphere of $\cala_p$ onto that of $\cala_2$ (\cite[Theorem 9.1]{bl}).

Assuming that the norm $\Vert v_n-Pv_n\Vert$ does not converge to 0, we deduce a contradiction.
For $n\in \N$, we set $u_n=v_n-Pv_n$ and $w_n=M_{p, 2}(u_n)$.
For any $n\in \N$, we have $u_n\in \calk_p$.
The sequence $(u_n)_n$ is asymptotically $G$-invariant under $\pi_p$ because for any $g\in G$, the equation $P\circ \pi_p(g)=\pi_p(g)\circ P$ holds.
Taking a subsequence of $(u_n)_n$, we may assume that the norm $\Vert u_n\Vert$ is uniformly positive.

The unit sphere of $\calk_p$ is uniformly separated from $\cala_p^C$ at distance 1, that is, for any unit vector $u\in \calk_p$ and any vector $v\in \cala_p^C$, we have $\Vert u-v\Vert \geq 1$.
One can check that the Mazur map $M_{p, 2}$ sends $\cala_p^C$ onto $\cala_2^C$, and for any $g\in G$, we have $M_{p, 2}\circ \pi_p(g)=\pi_2(g)\circ M_{p, 2}$.
It follows from the uniform continuity of $M_{p, 2}$ that the sequence $(w_n)_n$ is uniformly separated from $\cala_2^C$ and is asymptotically $G$-invariant under $\pi_2$.
For $n\in \N$, let $w_n'$ be the orthogonal projection of $w_n$ to the subspace $\calk_2$ that is the orthogonal complement of $\cala_2^C$ in $\cala_2$.
It also follows that the norm $\Vert w_n'\Vert$ is uniformly positive, and the sequence $(w_n')_n$ is asymptotically $G$-invariant under $\pi_2$.
The restriction of $\pi_2$ to $\calk_2$ thus has an asymptotically $G$-invariant sequence, but has no non-zero $C$-invariant vector.
The unitary representation of $G$ on the complexification of $\calk_2$ satisfies the same property.
This contradicts the assumption that the pair $(G, C)$ has property (T).
\end{proof}

We set $\calbb_p=L^p(X, \R)$, the Banach space of real-valued $p$-integrable functions on $(X, \mu)$.
Let $\kappa_p$ be the Koopman representation of $G$ on $\calbb_p$ associated with the action $G\c (X, \mu)$, i.e., the representation defined by
\[(\kappa_p(g)f)(x)=f(g^{-1}x)\quad \textrm{for}\ g\in G,\ f\in \calbb_p \ \textrm{and}\ x\in X.\]
We have the following description for $\kappa_p$ similar to that for $\pi_p$.
Let $\calbb_p^C$ denote the subspace of $\calbb_p$ of $C$-invariant functions on $X$.
We define a linear map $Q\colon \calbb_p \to \calbb_p$ by $Q(f)(x)=\mu_{\theta(x)}(f)$ for $f\in \calbb_p$ and $x\in X$.
The map $Q$ is the projection onto $\calbb_p^C$.
The space $\calbb_p$ is naturally identified with the subspace of $\cala_p$ of functions supported on $\{ e\} \times X$, where $e$ is the neutral element of $G$, so that $\kappa_p$ is a subrepresentation of $\pi_p$.
Under this identification, $Q$ is the restriction of $P$.
Lemma \ref{lem-p} therefore implies the following:

\begin{lem}\label{lem-q}
In the above notation, we suppose that the pair $(G, C)$ has property (T).
Let $(u_n)_{n=1}^\infty$ be a sequence of unit vectors in $\calbb_p$ which is asymptotically $G$-invariant under $\kappa_p$.
Then we have $\Vert u_n-Qu_n\Vert \to 0$ as $n\to \infty$.
\end{lem}


\subsection{Gaussian actions}

We briefly describe construction of the p.m.p.\ action associated with a unitary (or orthogonal) representation of a group, called the Gaussian action.
The reader should be referred to \cite[Appendix A.7]{bhv} and \cite[Appendices A--E]{kec} for more details.
Results in this subsection will be used in Section \ref{sec-ce} for the proof of Theorem \ref{thm-stab} (i).

Let $G$ be a discrete group, $\calh$ a complex Hilbert space with the inner product $\langle \cdot, \cdot \rangle_\calh$, and $(\pi, \calh)$ a unitary representation of $G$.
We fix an orthonormal basis $\Delta$ of $\calh$.
Let $\calh_\R$ be the realification of $\calh$.
This is the real Hilbert space whose elements are exactly those of $\calh$ and vector operations are defined by the restriction of the scalar field to $\R$.
The inner product on $\calh_\R$ is given by
\[\langle \xi, \eta \rangle_{\calh_\R}={\rm Re}\, \langle \xi, \eta \rangle_\calh \quad \textrm{for}\ \xi, \eta\in \calh_\R.\]
We set $\Delta_\R=\Delta \cup \{\,  i\xi \mid \xi \in \Delta \, \}$.
The set $\Delta_\R$ is an orthonormal basis of $\calh_\R$.
Each unitary operator on $\calh$ is also an orthogonal operator on $\calh_\R$.
We therefore have the orthogonal representation $(\pi_\R, \calh_\R)$ of $G$ associated with $\pi$.
We define the symmetric Fock space
\[S(\calh_\R)=\bigoplus_{k=0}^\infty S^k(\calh_\R),\]
where we set $S^0(\calh_\R)=\R$, and for each $k\in \N$, we denote by $S^k(\calh_\R)$ the $k$-th symmetric tensor product of $\calh_\R$.
We naturally have the orthogonal representation $(\pi_\R^S, S(\calh_\R))$ of $G$ associated with $\pi_\R$.

We define the standard probability space
\[(\Omega, \nu)=\prod_{\xi \in \Delta_\R}\left(\R, \frac{1}{\, \sqrt{2\pi}\, }e^{-x^2/2}\, dx\right).\]
For $\xi \in \Delta_\R$, we define a function $X_\xi\colon \Omega \to \R$ as the projection onto the coordinate $\xi$, i.e., the function given by $X_\xi((\omega_\eta)_{\eta \in \Delta_\R})=\omega_\xi$ for $(\omega_\eta)_{\eta \in \Delta_\R}\in \Omega$.
Let $L^2(\Omega, \nu)$ be the space of complex-valued square-integrable functions on $(\Omega, \nu)$.
Let $L^2(\Omega, \nu, \R)$ be the subspace of $L^2(\Omega, \nu)$ consisting of real-valued functions.

We have the canonical isometric isomorphism $\Psi \colon S(\calh_\R)\to L^2(\Omega, \nu, \R)$ (see the proof of \cite[Theorem A.7.13]{bhv}).
The map $\Psi$ sends a vector in $S^0(\calh_\R)$ to a constant function on $\Omega$ naturally.
For any $\xi \in \Delta_\R$, regarding it as a vector in $S^1(\calh_\R)$, we have $\Psi(\xi)=X_\xi$.
For any $T\in O(\calh_\R)$, there exists a unique automorphism $\bar{T}$ of $(\Omega, \nu)$ such that the orthogonal operator on $L^2(\Omega, \nu, \R)$ associated with $\bar{T}$ is isomorphic to the orthogonal operator on $S(\calh_\R)$ associated with $T$, through $\Psi$.
We thus have a p.m.p.\ action $G\c (\Omega, \nu)$ such that the associated representation of $G$ on $L^2(\Omega, \nu, \R)$ is isomorphic to the representation $\pi_\R^S$ of $G$ on $S(\calh_\R)$, through $\Psi$.
This action $G\c (\Omega, \nu)$ is called the {\it Gaussian action} associated with $\pi$.

Pick $\xi \in \Delta$ and set $A=\{ \, \omega \in \Omega \mid X_\xi(\omega)\geq 0\, \}$.
Assertion (i) in the following lemma asserts that for any $g\in G$, the measure $\nu(gA\bigtriangleup A)$ is close to 0 if and only if the vector $\pi(g)\xi$ is close to $\xi$.
This observation plays a crucial role in Connes-Weiss' construction of a non-trivial a.i.\ sequence (\cite{cw}), and also appears in the proof of \cite[Theorem 2.2.2]{ccjjv} and the proof of \cite[Theorem 2]{gw} (see also \cite[Theorem 6.3.4]{bhv} for Connes-Weiss' theorem).

\begin{lem}\label{lem-gauss}
Let $G$ be a discrete group, $\calh$ a complex Hilbert space with an orthonormal basis $\Delta$, and $(\pi, \calh)$ a unitary representation of $G$.
We have the p.m.p.\ action $G\c (\Omega, \nu)$ constructed above.
Let $\kappa$ denote the Koopman representation of $G$ on $L^2(\Omega, \nu)$ associated with this action.
Then the following assertions hold:
\begin{enumerate}
\item Pick $\xi \in \Delta$ and set $A=\{ \, \omega \in \Omega \mid X_\xi(\omega)\geq 0\, \}$.
Then $\nu(A)=1/2$, and for any $g\in G$, we have $\nu(gA\bigtriangleup A)=\alpha_g/\pi$,
where $\alpha_g={\rm Arccos}\, ({\rm Re}\, \langle \pi(g)\xi, \xi \rangle_\calh )$. 
\item Let $(g_n)_n$ be a sequence in $G$ such that for any $\eta \in \calh$, we have $\Vert \pi(g_n)\eta -\eta \Vert \to 0$ as $n\to \infty$.
Then for any $f\in L^2(\Omega, \nu)$, we have $\Vert \kappa(g_n)f-f\Vert \to 0$ as $n\to \infty$.
In particular, for any measurable subset $B$ of $\Omega$, we have $\nu(g_nB\bigtriangleup B)\to 0$ as $n\to \infty$.
\end{enumerate}
\end{lem}

\begin{proof}
We prove assertion (i).
The equation $\nu(A)=1/2$ follows from the origin symmetry of the measure $(1/\sqrt{2\pi})e^{-x^2/2}\, dx$ on $\R$.
Pick $g\in G$.
We have the equation
\[gX_\xi =\sum_{\eta \in \Delta_\R}\langle \pi_\R(g)\xi, \eta \rangle_{\calh_\R}X_\eta =({\rm Re}\, \langle \pi(g)\xi, \xi \rangle_\calh)X_\xi + \sum_{\eta \in \Delta_\R \setminus \{ \xi \}}\langle \pi_\R(g)\xi, \eta \rangle_{\calh_\R}X_\eta.\]
If $\alpha_g=0$, then $gX_\xi =X_\xi$.
We therefore have $gA=A$, and assertion (i) follows.
Suppose $\alpha_g\neq 0$.
The equation
\[gX_\xi =(\cos \alpha_g)X_\xi +(\sin \alpha_g)Y\]
then holds, where we set
\[Y=(\sin \alpha_g)^{-1}\sum_{\eta \in \Delta_\R \setminus \{ \xi \}}\langle \pi_\R(g)\xi, \eta \rangle_{\calh_\R}X_\eta.\]
Both $X_\xi$ and $Y$ are Gaussian random variables with mean 0 and variance 1, and they are independent.
By the circular symmetry of the joint distribution of $X_\xi$ and $Y$, we have
\[\nu(gA\bigtriangleup A)=\nu(\{  gX_\xi \geq 0,\ X_\xi <0 \})+\nu(\{  gX_\xi <0,\ X_\xi \geq 0 \})=\frac{\, \alpha_g}{2\pi }+\frac{\, \alpha_g}{ 2\pi }=\frac{\, \alpha_g}{\pi}.\]
Assertion (i) follows.

Under the assumption of assertion (ii), by the definition of the orthogonal representation $(\pi_\R^S, S(\calh_\R))$ of $G$, for any $\zeta \in S(\calh_\R)$, we have $\Vert \pi_\R^S(g_n)\zeta -\zeta \Vert \to 0$ as $n\to \infty$.
Assertion (ii) is proved through the $G$-equivariant isometric isomorphism $\Psi \colon S(\calh_\R)\to L^2(\Omega, \nu, \R)$.
\end{proof}



\section{Stable actions}\label{sec-js}

We discuss stability of a discrete measured groupoid and its characterization in terms of asymptotically central sequences, following Jones-Schmidt \cite{js}, who deal with the case where the groupoid is principal.
For simplicity of the notation, we only focus on groupoids associated with group actions.

\subsection{Notation and terminology}\label{subsec-stab}

Let $(X, \mu)$ be a standard probability space with $\calb$ the algebra of measurable subsets of $X$.
Let $G$ be a discrete group and $G\c (X, \mu)$ a p.m.p.\ action.
The discrete measured groupoid associated with this action is denoted by $G \ltimes (X, \mu)$, and is simply denoted by $G \ltimes X$ if there is no confusion.
This is a discrete measured groupoid on $(X, \mu)$ whose elements are exactly those of $G\times X$.
The range and source maps of $G \ltimes (X, \mu)$ are defined by $r(g, x)=gx$ and $s(g, x)=x$, respectively, for $g\in G$ and $x\in X$.
The product on $G \ltimes (X, \mu)$ is defined by $(g_1, g_2x)(g_2, x)=(g_1g_2, x)$ for $g_1, g_2\in G$ and $x\in X$.
The unit at $x\in X$ is $(e, x)$, where $e$ is the neutral element of $G$.
The inverse of $(g, x)$ for $g\in G$ and $x\in X$ is $(g^{-1}, gx)$.

We set $\calg =G\ltimes (X, \mu)$.
For a measurable map $U\colon X\to G$, let $U^0\colon X\to X$ denote the map defined by $U^0(x)=U(x)x$ for $x\in X$.
We define the {\it full group} of $\calg$, denoted by $[\calg]$, as the set of measurable maps $U\colon X\to G$ such that $U^0$ is an automorphism of $X$.
Let $I\in [\calg]$ be the map defined by $I(x)=e$ for any $x\in X$.
For $U, V\in [\calg]$, we define their composition $U\cdot V\in [\calg]$ by $(U\cdot V)(x)=U(V^0(x))V(x)$ for $x\in X$.
The equation $(U\cdot V)^0=U^0\circ V^0$ then holds.
For $U\in [\calg]$, we define its inverse $U^\dashv \in [\calg]$ by $U^\dashv(x)=U((U^0)^{-1}(x))^{-1}$ for $x\in X$.
The equation $U\cdot U^\dashv =I=U^\dashv \cdot U$ holds.
It follows that $[\calg]$ is a group with respect to these operations and with $I$ the neutral element.

A discrete measured equivalence relation $\calr$ on $(X, \mu)$ admits the structure of a discrete measured groupoid on $(X, \mu)$ defined as follows:
The range and source maps $r, s\colon \calr \to X$ are defined by $r(x, y)=x$ and $s(x, y)=y$ for $(x, y)\in \calr$.
The product on $\calr$ is defined by $(x, y)(y, z)=(x, z)$ for $(x, y), (y, z)\in \calr$.
The unit at $x\in X$ is $(x, x)$.
The inverse of $(x, y)\in \calr$ is $(y, x)$.
If $\calr$ is associated with the p.m.p.\ action $G\c (X, \mu)$, that is, $\calr$ is defined by the equation
\[\calr =\{ \, (gx, x)\in X\times X\mid g\in G,\ x\in X\, \},\]
then we have the homomorphism from $\calg$ into $\calr$ sending each $(g, x)\in \calg$ to $(gx, x)\in \calr$.
This homomorphism is an isomorphism if and only if the action $G\c (X, \mu)$ is free.

\medskip

\noindent {\bf Stability.}
Let $\calr_0$ be the ergodic hyperfinite equivalence relation of type ${\rm II}_1$ on a standard probability space $(X_0, \mu_0)$ (\cite[Chapter II]{km}, \cite[Chapter XIII]{take3}).
We have an ergodic, free and p.m.p.\ action $\Z \c (X_0, \mu_0)$ of the infinite cyclic group $\Z$ such that the associated equivalence relation is equal to $\calr_0$.

Let $G$ be a discrete group.
We say that an ergodic p.m.p.\ action $G\c (X, \mu)$ is {\it stable} if the associated groupoid $\calg =G\ltimes (X, \mu)$ is isomorphic to the direct product $\calg \times \calr_0$.
The groupoid $\calg \times \calr_0$ is associated with the coordinatewise action $G\times \Z \c (X\times X_0, \mu \times \mu_0)$.

We note that this definition of stability of an ergodic p.m.p.\ action is slightly different from Jones-Schmidt's one in \cite[Definition 3.1]{js}.
They define an ergodic p.m.p.\ action to be stable if the associated equivalence relation is isomorphic to its direct product with $\calr_0$.
Under the assumption that the action is free, their definition of stability is equivalent to ours.

\medskip

\noindent {\bf Two kinds of sequences.}
Let $G$ be a discrete group and $\sigma \colon G\c (X, \mu)$ a p.m.p.\ action.
We set $\calg =G\ltimes (X, \mu)$.
A sequence $(A_n)_{n=1}^\infty$ in $\calb$ is called {\it asymptotically invariant (a.i.)} for $\sigma$ if for any $U\in [\calg]$, we have $\mu(U^0A_n\bigtriangleup A_n)\to 0$ as $n\to \infty$.
An a.i.\ sequence $(A_n)_{n=1}^\infty$ for $\sigma$ is called {\it trivial} if we have $\mu(A_n)(1-\mu(A_n))\to 0$ as $n\to \infty$.

A sequence $(U_n)_{n=1}^\infty$ in $[\calg]$ is called {\it asymptotically central (a.c.)} for $\sigma$ if the following conditions (a) and (b) hold:
\begin{enumerate}
\item[(a)] For any $B\in \calb$, we have $\mu(U_n^0B\bigtriangleup B)\to 0$ as $n\to \infty$.
\item[(b)] For any $V\in [\calg]$, we have $\mu(\{\, x\in X\mid (U_n\cdot V)x\neq (V\cdot U_n)x\,\})\to 0$ as $n\to \infty$.
\end{enumerate}
An a.c.\ sequence $(U_n)_{n=1}^\infty$ for $\sigma$ is called {\it trivial} if we have $\mu(U_n^0A_n\bigtriangleup A_n)\to 0$ for any a.i.\ sequence $(A_n)_{n=1}^\infty$ for $\sigma$.

We note that a sequence $(A_n)_{n=1}^\infty$ in $\calb$ is a.i.\ for $\sigma$ if and only if for any $g\in G$, we have $\mu(gA_n\bigtriangleup A_n)\to 0$ as $n\to \infty$.
We also note that a sequence $(U_n)_{n=1}^\infty$ in $[\calg]$ is a.c.\ for $\sigma$ if and only if it satisfies condition (a) in the definition of an a.c.\ sequence and the following condition:
\begin{enumerate}
\item[(c)] For any $g\in G$, we have $\mu(\{\, x\in X\mid (U_n\cdot U_g)x\neq (U_g\cdot U_n)x\,\})\to 0$ as $n\to \infty$, where the map $U_g\in [\calg]$ is defined by $U_g(x)=g$ for any $x\in X$.
\end{enumerate}
These remarks are noticed in \cite[Section 2]{js} and \cite[Remark 3.3]{js}, respectively, when the action $\sigma$ is free.

Let $\calr$ be the equivalence relation associated with $\sigma$.
We define $[\calr]$ as the {\it full group} of $\calr$, i.e., the group of automorphisms of $X$ whose graphs are contained in $\calr$.
Let $\iota \colon [\calg]\to [\calr]$ be the surjective homomorphism defined by $\iota(U)=U^0$ for $U\in [\calg]$.
If the action $\sigma$ is free, then $\iota$ is injective, and thus an isomorphism.

Suppose that the action $\sigma$ is free.
We say that a sequence $(T_n)_{n=1}^\infty$ in $[\calr]$ is {\it asymptotically central (a.c.)} for $\sigma$ if the sequence $(\iota^{-1}(T_n))_{n=1}^\infty$ in $[\calg]$ is a.c.\ for $\sigma$.
We say that an a.c.\ sequence $(T_n)_{n=1}^\infty$ in $[\calr]$ for $\sigma$ is {\it trivial} if the a.c.\ sequence $(\iota^{-1}(T_n))_{n=1}^\infty$ in $[\calg]$ for $\sigma$ is trivial.
The reader should be reminded that under the assumption that the action $\sigma$ is free, this definition of an a.c.\ sequence in $[\calr]$ for $\sigma$ and its triviality is equivalent to that due to Jones-Schmidt (\cite[Definition 3.2]{js}).


\subsection{Characterization}

The aim of this subsection is to show the following:

\begin{thm}\label{thm-js}
An ergodic p.m.p.\ action of a discrete group is stable if and only if the action admits a non-trivial a.c.\ sequence in the full group of the associated groupoid.
\end{thm}

The theorem for a free action is due to Jones-Schmidt (\cite[Theorem 3.4]{js}).
The proof for a general case is essentially the same as theirs.
In fact, most part of our proof is verbatim after exchanging symbols of equivalence relations in their proof, into those of groupoids appropriately.
The exchanging is however not straightforward in several places.
To clarify the exchanging, we decide to write down a proof of Theorem \ref{thm-js}, using the same symbols as those in their proof whenever possible.

\begin{proof}[Proof of Theorem \ref{thm-js}]
The ``only if" part holds because the ergodic hyperfinite equivalence relation $\calr_0$ of type ${\rm II}_1$ admits a non-trivial a.c.\ sequence.
We prove the ``if" part.
Let $G$ be a discrete group and $\sigma \colon G\c (X, \mu)$ an ergodic p.m.p.\ action.
Let $\calb$ be the algebra of measurable subsets of $X$.
Set $\calg =G\ltimes (X, \mu)$.
We have a sequence $(V_n)_{n=1}^\infty$ in $[\calg]$ and an a.i.\ sequence $(A_n)_{n=1}^\infty$ for $\sigma$ satisfying the following conditions (\ref{ac1})--(\ref{ac3}):
\begin{equation}\label{ac1}
\textrm{For any $A\in \calb$, we have $\mu(V_n^0A\bigtriangleup A)\to 0$ as $n\to \infty$.}
\end{equation}
\begin{equation}\label{ac2}
\textrm{For any $U\in [\calg]$, we have $\mu(\{\, x\in X\mid (U\cdot V_n)x\neq (V_n\cdot U)x\,\})\to 0$ as $n\to \infty$.}
\end{equation}
\begin{equation}\label{ac3}
\textrm{The measure $\mu(V_n^0A_n\bigtriangleup A_n)$ does not converge to 0 as $n\to \infty$.}
\end{equation}
Taking a subsequence, we may assume that there exists a positive number $c$ such that
\[\mu(A_n\setminus V_n^0A_n)=\mu(V_n^0A_n\bigtriangleup A_n)/2\to c\ \textrm{as}\ n\to \infty.\]
We prove two lemmas to get an a.i.\ and a.c.\ sequence enjoying nice properties.
In the rest of the proof of Theorem \ref{thm-js}, we mean by an a.c.\ sequence for $\sigma$ that in $[\calg]$ unless otherwise mentioned.

\begin{lem}\label{lem-invol}
There exist an a.i.\ sequence $(B_n)_{n=1}^\infty$ for $\sigma$ and an a.c.\ sequence $(U_n)_{n=1}^\infty$ for $\sigma$ such that for any $n\in \N$, we have $U_n\cdot U_n=I$ and $U_n^0B_n=X\setminus B_n$.
\end{lem}

\begin{proof}
The proof consists of three steps.

\medskip

\noindent {\bf Step 1.}
We construct an a.i.\ sequence $(E_n)_{n=1}^\infty$ for $\sigma$ and an a.c.\ sequence $(W_n)_{n=1}^\infty$ for $\sigma$ such that we have $\mu(E_n)\to c$ as $n\to \infty$; and for any $n\in \N$, we have $W_n\cdot W_n=I$ and $W_n^0E_n\cap E_n=\emptyset$.

\medskip

For $n\in \N$, we set $E_n=A_n\setminus V_n^0A_n$.
By condition (\ref{ac2}), the sequence $(E_n)_n$ is a.i.\ for $\sigma$.
Define $W_n\in [\calg]$ by
\[W_nx=\begin{cases}
V_nx & \textrm{if}\ x\in E_n\\
(V_n)^\dashv x & \textrm{if}\ x\in V_n^0E_n\\
e & \textrm{otherwise},
\end{cases}\]
where $e$ is the neutral element of $G$ (see Figure \ref{fig-wn} (a)).
\begin{figure}
\begin{center}
\includegraphics[width=12cm]{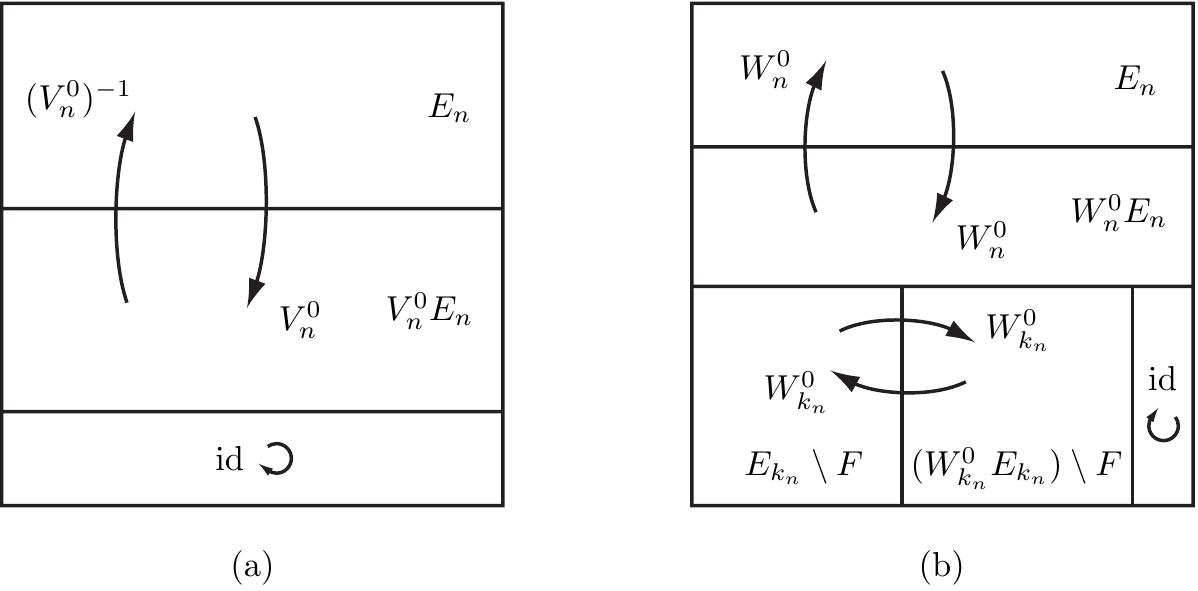}
\caption{(a) A description of $W_n^0$ with the whole square indicating $X$. (b) A description of $\tilde{W}_n^0$ with $F=(E_n\cup W_n^0E_n)\cup W_{k_n}^0(E_n\cup W_n^0E_n)$.}\label{fig-wn}
\end{center}
\end{figure}
The sequence $(W_n)_n$ is then a.c.\ for $\sigma$.
The properties in the last paragraph can also be checked.
Step 1 is completed.

\medskip

\noindent {\bf Step 2.}
We construct an a.i.\ sequence $(\tilde{E}_n)_{n=1}^\infty$ for $\sigma$ and an a.c.\ sequence $(\tilde{W}_n)_{n=1}^\infty$ for $\sigma$ such that we have $\mu(\tilde{E}_n)\to 2c-2c^2$ as $n\to \infty$; and for any $n\in \N$, we have $\tilde{W}_n\cdot \tilde{W}_n=I$ and $\tilde{W}_n^0\tilde{E}_n\cap \tilde{E}_n=\emptyset$.

\medskip

For any $n\in \N$, we can find $k_n\in \N$ with $k_n>n$ and the following inequalities (\ref{ichi})--(\ref{san}):
\begin{equation}\label{ichi}
|\mu(E_n\cap E_{k_n})-c\mu(E_n)|<\frac{1}{\, 2^n}.
\end{equation}
\begin{equation}\label{ni}
|\mu(W_n^0E_n\cap E_{k_n})-c\mu(W_n^0E_n)|<\frac{1}{\, 2^n}.
\end{equation}
\begin{equation}\label{san}
\mu((E_n\cup W_n^0E_n)\bigtriangleup W_{k_n}(E_n\cup W_n^0E_n))<\frac{1}{\, 2^n}.
\end{equation}
Inequalities (\ref{ichi}) and (\ref{ni}) follow from that $\sigma$ is ergodic and we have $\mu(E_n)\to c$ as $n\to \infty$ (see the proof of \cite[Lemma 2.3]{js}).
Inequality (\ref{san}) follows from condition (\ref{ac1}) for the a.c.\ sequence $(W_n)_n$ in place of $(V_n)_n$.
For $n\in \N$, we set
\[\tilde{E}_n=E_n\cup (E_{k_n}\setminus ((E_n\cup W_n^0E_n)\cup W_{k_n}(E_n\cup W_n^0E_n))),\]
and define $\tilde{W}_n\in [\calg]$ by
\[\tilde{W}_nx=\begin{cases}
W_nx & \textrm{if}\ x\in E_n\cup W_n^0E_n\\
W_{k_n}x & \textrm{if}\ x\in (E_{k_n}\cup W_{k_n}^0E_{k_n})\setminus ((E_n\cup W_n^0E_n)\cup W_{k_n}^0(E_n\cup W_n^0E_n))\\
e & \textrm{otherwise}
\end{cases}\]
(see Figure \ref{fig-wn} (b)).
We then see that $(\tilde{E}_n)_n$ is a.i.\ for $\sigma$; $(\tilde{W}_n)_n$ is a.c.\ for $\sigma$; and for any $n\in \N$, we have $\tilde{W}_n\cdot \tilde{W}_n=I$ and $\tilde{W}_n^0\tilde{E}_n\cap \tilde{E}_n=\emptyset$.
We have $\mu(\tilde{E_n})\to 2c-2c^2$ as $n\to \infty$ by inequalities (\ref{ichi})--(\ref{san}).
Step 2 is completed.

\medskip

\noindent {\bf Step 3.}
We construct an a.i.\ sequence $(B_n)_{n=1}^\infty$ for $\sigma$ and an a.c.\ sequence $(U_n)_{n=1}^\infty$ for $\sigma$ such that for any $n\in \N$, we have $U_n\cdot U_n=I$ and $U_n^0B_n=X\setminus B_n$.

\medskip

For a real number $d$ with $0<d<1/2$, define a sequence $(d_n)_{n=1}^\infty$ of real numbers by $d_1=d$ and $d_{n+1}=2d_n-2d_n^2$.
We then have $d_n\nearrow 1/2$ as $n\to \infty$.
It follows from Step 2 that we can find an a.i.\ sequence $(C_n)_{n=1}^\infty$ for $\sigma$ and an a.c.\ sequence $(S_n)_{n=1}^\infty$ for $\sigma$ such that for any $n\in \N$, we have
\[\frac{1}{\, 2\, }-\frac{1}{\, 2^n}<\mu(C_n)\leq \frac{1}{\, 2\, },\quad S_n^0C_n\cap C_n=\emptyset \quad \textrm{and}\quad S_n\cdot S_n=I.\]
For $n\in \N$, we choose $B_n\in \calb$ with $\mu(B_n)=1/2$, $C_n\subset B_n$ and $B_n\cap S_n^0C_n=\emptyset$.
We also choose $U_n\in [\calg]$ such that $U_n|_{C_n}=S_n$, $U_n\cdot U_n=I$ and $U_n^0B_n=X\setminus B_n$.
This $U_n$ can be chosen because the action $\sigma$ is ergodic.
The sequences $(B_n)_n$ and $(U_n)_n$ are desired ones.
Step 3 is completed, and Lemma \ref{lem-invol} is therefore proved.
\end{proof}

\begin{lem}\label{lem-tile}
There exist an a.i.\ sequence $(D_n)_{n=1}^\infty$ for $\sigma$ and an a.c.\ sequence $(V_n)_{n=1}^\infty$ for $\sigma$ satisfying the following conditions {\rm (\ref{a})--(\ref{c})}:
\begin{equation}\label{a}
\textrm{For any $n\in \N$, we have $V_n\cdot V_n=I$ and $V_n^0D_n=X\setminus D_n$.}
\end{equation}
\begin{equation}\label{b}
\textrm{For any distinct $n, m\in \N$, we have $V_n\cdot V_m=V_m\cdot V_n$ and $V_n^0D_m=D_m$.}
\end{equation}
\begin{eqnarray}\label{c}
\qquad \textrm{For any $k\in \N$ and any mutually distinct $n_1,\ldots, n_k\in \N$,}\\ 
\textrm{we have $\mu(D_{n_1}\cap \cdots \cap D_{n_k})=1/2^k$.}\quad \nonumber
\end{eqnarray}
\end{lem}

\begin{proof}
We construct $D_n$ and $V_n$ inductively.
Let $(B_n)_{n=1}^\infty \subset \calb$ and $(U_n)_{n=1}^\infty \subset [\calg]$ be the sequences in Lemma \ref{lem-invol}.
We set $D_1=B_1$ and $V_1=U_1$.

Suppose that we have $D_1,\ldots, D_N$ and $V_1,\ldots, V_N$ satisfying conditions (\ref{a})--(\ref{c}) for them.
We now construct $D_{N+1}$ and $V_{N+1}$.
Let $\Gamma_N$ denote the subgroup of $[\calg]$ generated by $V_1,\ldots, V_N$, which is isomorphic to the direct product of $N$ copies of $\Z /2\Z$.
Since $(B_n)_n$ is a.i.\ for $\sigma$ and $(U_n)_n$ is a.c.\ for $\sigma$, there exists $k\in \N$ with $k>N$ satisfying the following conditions (\ref{d})--(\ref{f}):
\begin{equation}\label{d}
\textrm{For any $V\in \Gamma_N$, we have $\mu(\{ \, x\in X\mid (U_k\cdot V)x\neq (V\cdot U_k)x\, \})<\frac{1}{\, 4^{N+1}}$,}
\end{equation}
\begin{equation}\label{e}
\textrm{For any $j=1,\ldots, N$, we have $\mu(U_k^0D_j\bigtriangleup D_j)<\frac{1}{\, 4^{N+1}}$,}
\end{equation}
\begin{equation}\label{f}
\textrm{For any $V\in \Gamma_N$, we have $\mu(V^0B_k\bigtriangleup B_k)<\frac{1}{\, 4^{N+1}}$}.
\end{equation}
We set
\[E_1=\bigcup_{V\in \Gamma_N}\{ \, x\in X\mid (U_k\cdot V)x\neq (V\cdot U_k)x\, \} \quad \textrm{and}\quad E_2=\bigcup_{j=1}^N(U_k^0D_j\bigtriangleup D_j).\]

\begin{claim}\label{claim-e}
We set $E=E_1\cup E_2$.
Then the following assertions hold:
\begin{enumerate}
\item[(i)] For any $T\in \Gamma_N$, we have $T^0E=E$.
\item[(ii)] We have $U_k^0E=E$ and $\mu(E)<1/2^{N+1}$.
\end{enumerate}
\end{claim}

\begin{proof}
We prove assertion (i).
Fix $T\in \Gamma_N$.
We first prove the equation $T^0E_1=E_1$, and then prove the equation $T^0E=E$.
For any $x\in X\setminus E_1$ and any $V\in \Gamma_N$, we have
\begin{align*}
((U_k\cdot V)(T^0x))(Tx) & =(U_k\cdot V\cdot T)x=(V\cdot T\cdot U_k)x=(V \cdot U_k\cdot T)x\\
& =((V\cdot U_k)(T^0x))(Tx)
\end{align*}
because $V\cdot T$ belongs to $\Gamma_N$.
We thus have $T^0x\in X\setminus E_1$, and have $T^0E_1=E_1$.

Pick $y\in X\setminus E$.
The equation $T^0E_1=E_1$ implies $T^0y\in X\setminus E_1$.
If we had $T^0y\in E_2$, then there would exist $j=1,\ldots, N$ with $T^0y\in U_k^0D_j\bigtriangleup D_j$.
By conditions (\ref{a}) and (\ref{b}) for $V_1,\ldots, V_N$ and $D_1,\ldots, D_N$, either $T^0D_j=D_j$ or $T^0D_j=X\setminus D_j$ holds.

We first assume $T^0y\in U_k^0D_j\setminus D_j$, and deduce a contradiction.
The equation $(U_k\cdot T)y=(T\cdot U_k)y$ holds because $y\in X\setminus E_1$.
We have $(U_k^0\circ T^0\circ U_k^0)y=T^0y\in U_k^0D_j$, and thus $y\in (U_k^0\circ T^0)D_j$.
It follows that we have $y\in U_k^0D_j$ if $T^0D_j=D_j$, and we have $y\in X\setminus U_k^0D_j$ otherwise.
It follows from $T^0y\in X\setminus D_j$ that we have $y\in X\setminus D_j$ if $T^0D_j=D_j$, and we have $y\in D_j$ otherwise.
In either case, we have $y\in U_k^0D_j\bigtriangleup D_j$, and thus $y\in E_2$.
This is a contradiction.
Assuming $T^0y\in D_j\setminus U_k^0D_j$ in place of the condition $T^0y\in U_k^0D_j\setminus D_j$, we can deduce a contradiction similarly.

We proved the equation $T^0E=E$.
Assertion (i) follows.

We prove assertion (ii).
By the equation $U_k\cdot U_k=I$, we have $U_k^0E_1=E_1$ and $U_k^0E_2=E_2$.
The inequality $\mu(E)<1/2^{N+1}$ follows from inequalities (\ref{d}) and (\ref{e}).
Assertion (ii) follows.
\end{proof}

Claim \ref{claim-e} (ii) asserts that the set $E$ is very small.
We will define $V_{N+1}\in [\calg]$ so that it is equal to $U_k$ on the complement $X\setminus E$, on which $U_k$ and any $V\in \Gamma_N$ commute.
To get desired properties of $V_{N+1}$, we modify the map $U_k$ on $E$.

We set $A_0=D_1\cap \cdots \cap D_N$.
Pick $W\in [\calg]$ and a measurable subset $\tilde{B}$ of $A_0\cap E$ such that
\[W\cdot W=I,\quad W^0(A_0\cap E)=A_0\cap E\quad \textrm{and}\quad W^0\tilde{B}=(A_0\cap E)\setminus \tilde{B}\]
(see Figure \ref{fig-aw}).
\begin{figure}
\begin{center}
\includegraphics[width=12cm]{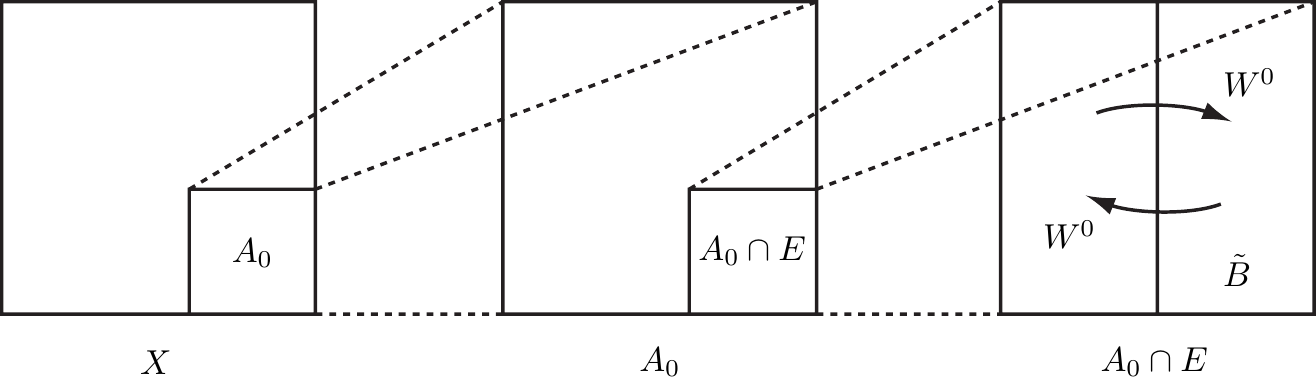}
\caption{}\label{fig-aw}
\end{center}
\end{figure}
Let $\Gamma_N^0$ denote the group of automorphisms $T^0$ of $X$ with $T\in \Gamma_N$.
We set $\tilde{B}_k=\Gamma_N^0(B_k\cap A_0)$, and set
\[D_{N+1}=(\tilde{B}_k\setminus E)\cup \Gamma_N^0\tilde{B}.\]
We define $V_{N+1}\in [\calg]$ by
\[V_{N+1}x=\begin{cases}
U_kx & \textrm{if}\ x\in X\setminus E\\
(V\cdot W\cdot V)x & \textrm{if}\ x\in V^0(A_0\cap E)\ \textrm{for\ some}\ V\in \Gamma_N.
\end{cases}\]
Inequality (\ref{f}) implies that
\[\mu(\tilde{B}_k\bigtriangleup B_k)\leq \sum_{V\in \Gamma_N}\mu(V^0B_k\bigtriangleup B_k)<\frac{1}{\, 2^{N+1}},\]
and thus
\begin{equation}\label{g}
\mu(D_{N+1}\bigtriangleup B_k)\leq \mu(D_{N+1}\bigtriangleup \tilde{B}_k)+\mu(\tilde{B}_k\bigtriangleup B_k)<2\mu(E)+\frac{1}{\, 2^{N+1}}<\frac{1}{\, 2^{N-1}}.
\end{equation}
By the definition of $V_{N+1}$, we see that $V_{N+1}\cdot V_{N+1}=I$ and that for any $m=1,\ldots, N$, we have $V_{N+1}\cdot V_m=V_m\cdot V_{N+1}$.
Since each of the sets $\tilde{B}_k$, $E$ and $\Gamma_N^0\tilde{B}$ is invariant under any element of $\Gamma_N^0$, for any $m=1,\ldots, N$, we have $V_m^0D_{N+1}=D_{N+1}$.

We next show the equation $V_{N+1}^0D_{N+1}=X\setminus D_{N+1}$.
We set
\[Y=X\setminus E\quad \textrm{and}\quad C=\tilde{B}_k\setminus E.\]
The automorphism $V_{N+1}^0$ of $X$ preserves each of $Y$ and $E$, and the inclusions $C\subset Y$ and $\Gamma_N^0\tilde{B}\subset E$ hold.
By the definition of $V_{N+1}^0$, the equation $V_{N+1}^0(\Gamma_N^0\tilde{B})=E\setminus (\Gamma_N^0\tilde{B})$ holds.
It suffices to show the equation $V_{N+1}^0C=Y\setminus C$.

For any $j=1,\ldots, N$, the inclusion $U_k^0D_j\setminus E=U_k^0(D_j\setminus E)\subset D_j\setminus E$ holds, where the first equation follows from Claim \ref{claim-e} (ii), and the last inclusion holds because $E_2\subset E$.
Since $U_k^0\circ U_k^0$ is the identity on $X$, we obtain the equation $U_k^0D_j\setminus E=D_j\setminus E$.
It follows that $U_k^0$ preserves $A_0\setminus E=A_0\cap Y$.

Recall the equations $U_k^0B_k=X\setminus B_k$ and $C\cap A_0=B_k\cap A_0\cap Y$.
We thus obtain the equation $U_k^0(C\cap A_0)=(A_0\cap Y)\setminus (C\cap A_0)$.
The map $V_{N+1}$ is equal to $U_k$ on $Y$, and commutes with any element of $\Gamma_N$.
Since $A_0\cap Y$ is a fundamental domain for the action of $\Gamma_N^0$ on $Y$, we obtain the equation $V_{N+1}^0C=Y\setminus C$.

We have shown the equation $V_{N+1}^0D_{N+1}=X\setminus D_{N+1}$.
Condition (\ref{c}) for $D_1,\ldots, D_{N+1}$ holds because $\Gamma_N^0D_{N+1}=D_{N+1}$ and $\mu(D_{N+1}\cap A_0)=1/2^{N+1}$.

We have constructed $D_1,\ldots, D_{N+1}$ and $V_1,\ldots, V_{N+1}$ satisfying conditions (\ref{a})--(\ref{c}).
The induction is completed.
We obtain a sequence $(D_n)_{n=1}^\infty$ in $\calb$ and a sequence $(V_n)_{n=1}^\infty$ in $[\calg]$ satisfying conditions (\ref{a})--(\ref{c}).
Inequality (\ref{g}) and the assumption that $(B_n)_n$ is a.i.\ for $\sigma$ imply that $(D_n)_n$ is also a.i.\ for $\sigma$.
Recall that $V_{N+1}$ is defined so that it is equal to $U_k$ on $X\setminus E$.
The inequality $\mu(E)<1/2^{N+1}$ and the assumption that $(U_n)_n$ is a.c.\ for $\sigma$ imply that $(V_n)_n$ is also a.c.\ for $\sigma$.
Lemma \ref{lem-tile} is proved.
\end{proof}

To get a decomposition of $\calg$ into a direct product, we show the following:

\begin{lem}\label{lem-t}
Let $(D_n)_{n=1}^\infty \subset \calb$ and $(V_n)_{n=1}^\infty \subset [\calg]$ be the sequences in Lemma \ref{lem-tile}.
Then there exist an increasing sequence of positive integers, $n_1<n_2<n_3<\cdots$, and a sequence $(T_m)_{m=1}^\infty$ in $[\calg]$ satisfying the following conditions {\rm (\ref{h})--(\ref{k})}:
\begin{eqnarray}\label{h}
\qquad \textrm{For any $m\in \N$, any $\gamma \in \Gamma$ and any $D\in \cal{D}$,}\qquad \qquad \qquad \\
\textrm{we have $T_m\cdot \gamma =\gamma \cdot T_m$ and $T_m^0D=D$,}\nonumber
\end{eqnarray}
where we define $\Gamma$ as the subgroup of $[\calg]$ generated by $V_{n_1}, V_{n_2}, V_{n_3},\ldots$, and define $\cal{D}$ as the $\sigma$-subalgebra of $\calb$ generated by $D_{n_1}, D_{n_2}, D_{n_3},\ldots$. 
\begin{equation}\label{i}
\textrm{For a.e.\ $x\in X$, we have $G=\{ \, (T_m\cdot \gamma)x \mid m\in \N,\ \gamma \in \Gamma \, \}$.}
\end{equation}
\begin{equation}\label{j}
\textrm{The $\sigma$-algebra $\calb$ is generated by $\cal{D}$ and $\calb^{\Gamma}$,}
\end{equation}
where we set $\calb^\Gamma =\{ \, B\in \calb \mid \gamma^0B=B\ {\rm for\ any}\ \gamma \in \Gamma \, \}$.
\begin{equation}\label{k}
\textrm{For any $B\in \calb^\Gamma$ and any $D\in \cal{D}$, we have $\mu(B\cap D)=\mu(B)\mu(D)$.}
\end{equation}
\end{lem}

\begin{proof}
The numbers $n_1<n_2<n_3<\cdots$ and $T_m\in [\calg]$ in the lemma are constructed inductively.
We fix the notation as follows:
We set $\alpha_0=\{ X\} \subset \calb$.
For $k\in \N$, let $\alpha_k$ be the collection of minimal elements in the subalgebra of $\calb$ generated by $D_{n_1}, D_{n_2},\ldots, D_{n_k}$.
Let $A_{k, 1}, A_{k, 2},\ldots, A_{k, 2^k}$ denote the elements of $\alpha_k$.
Let $\Gamma_0$ be the trivial subgroup of $[\calg]$, and $\Gamma_k$ the subgroup of $[\calg]$ generated by $V_{n_1}, V_{n_2},\ldots, V_{n_k}$.
We set
\[\calb_k =\{ \, B\in \calb \mid V^0B=B\ {\rm for\ any}\ V\in \Gamma_k \, \}.\]

\medskip

\noindent {\bf Construction of $W^{(n)}$.}
Fix $k\in \N$.
Pick $W\in [\calg]$ such that $W^0\circ W^0$ is the identity on $X$; for any $V\in \Gamma_k$, we have $W\cdot V=V\cdot W$; and for any $A\in \alpha_k$, we have $W^0A=A$.
For $n\in \N$ with $n>n_k$, we define $W^{(n)}\in [\calg]$ by
\[W^{(n)}x=\begin{cases}
Wx & \textrm{if}\ x\in W^0D_n\cap D_n\\
(V_n\cdot W\cdot V_n)x & \textrm{if}\ x\in V_n^0(W^0D_n\cap D_n)\\
e & \textrm{otherwise}
\end{cases}\]
(see Figure \ref{fig-w-bra-n}).
\begin{figure}
\begin{center}
\includegraphics[width=6cm]{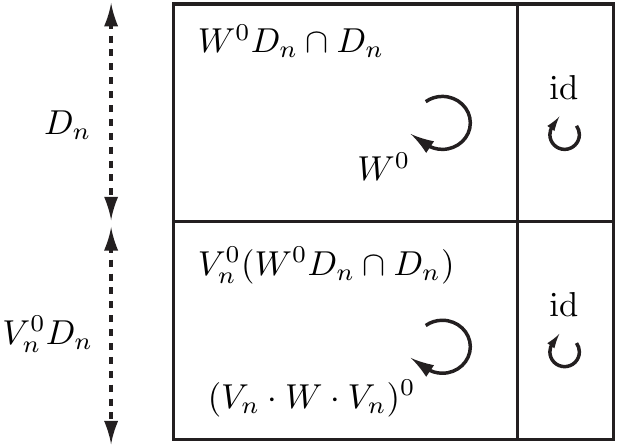}
\caption{A description of $(W^{(n)})^0$ with the whole square indicating $X$. The automorphism $V_n^0$ of $X$ exchanges $D_n$ and $V_n^0D_n$.}\label{fig-w-bra-n}
\end{center}
\end{figure}
Notable properties of $W^{(n)}$ are listed in the following:

\begin{claim}\label{claim-wn}
The following assertions hold:
\begin{enumerate}
\item The automorphism $(W^{(n)})^0\circ (W^{(n)})^0$ is the identity on $X$.
\item For any $V\in \Gamma_k\cup \{ V_n\}$, we have $W^{(n)}\cdot V=V\cdot W^{(n)}$.
\item For any $A\in \alpha_k\cup \{ D_n\}$, we have $(W^{(n)})^0A=A$.
\item We have $\mu(\{ \, x\in X\mid W^{(n)}x\neq Wx\, \})\to 0$ as $n\to \infty$.
\end{enumerate}
\end{claim}

\begin{proof}
Assertion (i) follows from the definition of $W^{(n)}$.

We prove assertion (ii).
Pick $V\in \Gamma_k$.
Since $n>n_k$, we have $V^0D_n=D_n$ by condition (\ref{b}).
The equation $W\cdot V=V\cdot W$ implies that $V^0$ preserves $W^0D_n\cap D_n$.
The equation $V_n\cdot V=V\cdot V_n$ implies that $V^0$ also preserves $V_n^0(W^0D_n\cap D_n)$.
We can now check the equation $W^{(n)}\cdot V=V\cdot W^{(n)}$ from the definition of $W^{(n)}$.
Commutativity of $W^{(n)}$ and $V_n$ follows from the definition of $W^{(n)}$ (see Figure \ref{fig-w-bra-n}).
Assertion (ii) is proved.

We prove assertion (iii).
Pick $A\in \alpha_k$.
The automorphism $W^0$ preserves each of $A$ and $W^0D_n\cap D_n$.
It follows that $(W^{(n)})^0$ preserves the set $A\cap W^0D_n\cap D_n$.
Using the equation $V_n^0A=A$ in condition (\ref{b}), we see that $(W^{(n)})^0$ preserves the set $A\cap V_n^0(W^0D_n\cap D_n)$.
It therefore follows that $(W^{(n)})^0$ preserves $A$.
We have the equation $(W^{(n)})^0D_n=D_n$ by the definition of $W^{(n)}$ (see Figure \ref{fig-w-bra-n}).
Assertion (iii) is proved.

Assertion (iv) holds because the sequence $(D_n)_n$ is a.i.\ for $\sigma$ and the sequence $(V_n)_n$ is a.c.\ for $\sigma$.
\end{proof}

Let $G=\{ g_0=e, g_1, g_2,\ldots \}$ be an enumeration of elements of $G$.
We fix a countable subset $\cal{F}=\{ F_1, F_2, F_3,\ldots \}$ of $\calb$ generating $\calb$.
Following Jones-Schmidt \cite[pp.106--108]{js} and using Claim \ref{claim-wn}, we can find
\begin{itemize}
\item an increasing sequence of positive integers, $n_1<n_2<n_3<\cdots$;
\item a positive integer $j(p, q)$ for any $p, q\in \N$ with $q\leq p$; and
\item $W_{p, q, r, t}\in [\calg]$ and $F_{p, q, s, t}\in \calb_t$ for any integers $p$, $q$, $r$, $s$ and $t$ with $1\leq q\leq p$, $1\leq r\leq j(p, q)$, $1\leq s\leq 2^{p-1}$ and $p\leq t$
\end{itemize}
satisfying the following conditions (\ref{wf1})--(\ref{wf6}):
For any integers $p$, $q$, $r$, $s$, $t$ and $t'$ with $1\leq q\leq p$, $1\leq r\leq j(p, q)$, $1\leq s\leq 2^{p-1}$, $p\leq t$ and $p\leq t'$, for any $V\in \Gamma_t$ and for any $A\in \alpha_t$, we have
\begin{equation}\label{wf1}
\mu(\{ \, x\in X\mid g_q\not\in \{ \, (V'\cdot W_{p, q, r', p})x \mid V'\in \Gamma_{p-1},\ 1\leq r'\leq j(p, q)\, \} \, \})<\frac{1}{\, 2^p},
\end{equation}
\begin{equation}\label{wf2}
\mu(\{ \, x\in X\mid W_{p, q, r, t'}x\neq W_{p, q, r, t'+1}x \, \})<\frac{1}{\, 4^{t'+1}j(p, q)},
\end{equation}
\begin{equation}\label{wf3}
W_{p, q, r, t}\cdot V=V\cdot W_{p, q, r, t},
\end{equation}
\begin{equation}\label{wf4}
W_{p, q, r, t}^0A=A,
\end{equation}
\begin{equation}\label{wf5}
\mu((F_{p, q, s, p}\bigtriangleup F_q)\cap A_{p-1, s})<\frac{1}{\, 4^p},
\end{equation}
\begin{equation}\label{wf6}
\mu(F_{p, q, s, t'}\bigtriangleup F_{p, q, s, t'+1})<\frac{1}{\, 4^{t'+1}}.
\end{equation}
The process to find such integers, $n_1<n_2<n_3<\cdots$, etc.\ is a verbatim translation of Jones-Schmidt's argument.
We thus omit it.

Let $p$, $q$ and $r$ be integers with $1\leq q\leq p$ and $1\leq r\leq j(p, q)$.
By inequality (\ref{wf2}), the limit of $W_{p, q, r, t}$ as $t\to \infty$ exists, and we denote it by $\tilde{W}_{p, q, r}\in [\calg]$.
The limit is taken with respect to the metric $d$ on $[\calg]$ defined by $d(U, V)=\mu(\{ \, x\in X\mid Ux\neq Vx\, \})$ for $U, V\in [\calg]$.
By equation (\ref{wf3}), $\tilde{W}_{p, q, r}$ commutes with any element of $\Gamma$.
By equation (\ref{wf4}), $\tilde{W}_{p, q, r}^0$ fixes any element of $\cal{D}$.
Inequality (\ref{wf2}) also implies the inequality
\[\mu(\{ \, x\in X\mid \tilde{W}_{p, q, r}x\neq W_{p, q, r, p}x\, \} )<\sum_{t=p+1}^\infty \frac{1}{\, 4^tj(p, q)}=\frac{1}{\, 3\cdot 4^pj(p, q)}.\]
By inequality (\ref{wf1}), for any integers $p$, $q$ with $1\leq q\leq p$, we have
\[\mu(\{ \, x\in X\mid g_q\not\in \{ \, (V\cdot \tilde{W}_{p, q, r})x \mid V\in \Gamma_{p-1},\ 1\leq r\leq j(p, q)\, \} \, \})<\frac{1}{\, 2^p}+\frac{1}{\, 3\cdot 4^p}<\frac{1}{\, 2^{p-1}}.\]
Let $(T_m)_{m=1}^\infty$ be a sequence obtained by ordering elements of the set
\[\{ \, \tilde{W}_{p, q, r}\mid 1\leq q\leq p,\ 1\leq r\leq j(p, q)\, \}.\]
Conditions (\ref{h}) and (\ref{i}) for this $(T_m)_{m=1}^\infty$ are checked above.

Let $p$, $q$ and $s$ be integers with $1\leq q\leq p$ and $1\leq s\leq 2^{p-1}$.
By inequality (\ref{wf6}), the limit of $F_{p, q, s, t}$ as $t\to \infty$ exists, and we denote it by $\tilde{F}_{p, q, s}\in \calb$.
The limit is taken with respect to the metric $d'$ on $\calb$ defined by $d'(A, B)=\mu(A\bigtriangleup B)$ for $A, B\in \calb$.
Since $F_{p, q, s, t}$ belongs to $\calb_t$ for any $t$, the limit $\tilde{F}_{p, q, s}$ belongs to $\calb^\Gamma$.
Inequality (\ref{wf6}) also implies the inequality
\[\mu(\tilde{F}_{p, q, s}\bigtriangleup F_{p, q, s, p})<\sum_{t=p+1}^\infty \frac{1}{\, 4^t}=\frac{1}{\, 3\cdot 4^p}.\]
By inequality (\ref{wf5}), for any integers $p$, $q$ with $1\leq q\leq p$, there exists a set $C$ in the subalgebra of $\calb$ generated by $\alpha_{p-1}$ and $\calb^\Gamma$ such that $\mu(F_q\bigtriangleup C)<2^{p-1}(1/4^p+1/(3\cdot 4^p))<1/2^p$.
Condition (\ref{j}) follows.
For any $B\in \calb^\Gamma$, any $k\in \N$ and any element $D$ of the subalgebra of $\calb$ generated by $D_{n_1}, D_{n_2},\ldots, D_{n_k}$, the equation $\mu(B\cap D)=\mu(B)\mu(D)$ holds.
Condition (\ref{k}) follows.
Lemma \ref{lem-t} is proved.
\end{proof}

We fix the notation.
We define $\Theta$ as the subgroup of $[\calg]$ generated by $(T_m)_{m=1}^\infty$, and set
\[\calb^\Theta =\{ \, B\in \calb \mid T^0B=B\ {\rm for\ any}\ T\in \Theta \, \}.\]
Let $\Theta^0$ denote the group of automorphisms $T^0$ of $X$ with $T\in \Theta$.
Similarly, let $\Gamma^0$ denote the group of automorphisms $\gamma^0$ of $X$ with $\gamma \in \Gamma$. 
Let $f_1\colon (X, \mu)\to (X_1, \mu_1)$ be the ergodic decomposition for the action of $\Theta^0$ on $(X, \mu)$.
Let $f_2\colon (X, \mu)\to (X_2, \mu_2)$ be the ergodic decomposition for the action of $\Gamma^0$ on $(X, \mu)$.
We define a map $f\colon X\to X_1\times X_2$ by $f(x)=(f_1(x), f_2(x))$ for $x\in X$.

\begin{lem}\label{lem-b}
The following assertions hold:
\begin{enumerate}
\item The equations $\calb^\Gamma \cap \calb^\Theta =\{ \emptyset, X\}$ and $\calb^\Theta =\cal{D}$ hold.
\item The map $f\colon (X, \mu)\to (X_1\times X_2, \mu_1\times \mu_2)$ is an isomorphism of measure spaces, that is, the restriction of $f$ to some conull measurable subset of $X$ is an isomorphism onto a conull measurable subset of $X_1\times X_2$, and the equation $f_*\mu=\mu_1\times \mu_2$ holds.
\end{enumerate}
\end{lem}

\begin{proof}
We prove assertion (i).
Condition (\ref{i}) and ergodicity of the action $\sigma \colon G\c (X, \mu)$ imply that $\calb^\Gamma \cap \calb^\Theta =\{ \emptyset, X\}$.

The inclusion $\cal{D}\subset \calb^\Theta$ follows from condition (\ref{h}).
To prove the converse inclusion, we fix the notation.
Let $\mu =\int_{X_1}\mu_z \, d\mu_1(z)$ be the disintegration of $\mu$ with respect to $f_1$.
For an integrable function $\varphi$ on $(X, \mu)$, we define a function $\bar{\varphi}$ on $X$ by $\bar{\varphi}(x)=\mu_{f_1(x)}(\varphi)$ for $x\in X$.
The function $\bar{\varphi}$ is then $\Theta^0$-invariant.
For $B\in \calb$, let $\chi_B\colon X\to \{ 0, 1\}$ denote the characteristic function on $B$.

Pick $A\in \calb^\Theta$.
Recall that for $k\in \N$, the subsets of $X$, $A_{k, 1}, A_{k, 2},\ldots, A_{k, 2^k}$, denote the minimal elements in the subalgebra of $\calb$ generated by $D_{n_1}, D_{n_2},\ldots, D_{n_k}$.
By condition (\ref{j}), the set $A$ is approximated by a subset of $X$ of the form
\[E=\bigsqcup_{j=1}^{2^k}(A_{k, j}\cap C_j)\]
with $k\in \N$ and $C_j \in \calb^\Gamma$.
The function $\chi_A=\bar{\chi}_A$ is therefore approximated by the function $\bar{\chi}_E$.
Pick an integer $j$ with $1\leq j\leq 2^k$, and put $E_j=A_{k, j}\cap C_j$.
Since $A_{k, j}$ is $\Theta^0$-invariant, we have $\bar{\chi}_{E_j}(x)=\mu_{f_1(x)}(\chi_{C_j})$ for a.e.\ $x\in A_{k, j}$, and have $\bar{\chi}_{E_j}(x)=0$ for a.e.\ $x\in X\setminus A_{k, j}$.
We thus have $\bar{\chi}_{E_j}=\bar{\chi}_{C_j}$ on $A_{k, j}$.
On the other hand, the function $\bar{\chi}_{C_j}$ is constant on $X$ because it is invariant under both $\Gamma^0$ and $\Theta^0$.
It follows that the function $\bar{\chi}_E$ is constant on $A_{k, j}$ with the value $\mu(C_j)$, and is therefore measurable with respect to $\cal{D}$.
Since $\chi_A$ is approximated by $\bar{\chi}_E$, the set $A$ belongs to $\cal{D}$.
The inclusion $\calb^\Theta \subset \cal{D}$ is proved.
Assertion (i) follows.

We prove assertion (ii).
The equation $f_*\mu=\mu_1\times \mu_2$ follows from condition (\ref{k}) and the equations $(f_1)_*\mu=\mu_1$, $(f_2)_*\mu =\mu_2$ and $\calb^\Theta =\cal{D}$.
Let $f^*$ be the map from the algebra of measurable subsets of $X_1\times X_2$ into $\calb$ associated to $f$.
By condition (\ref{j}) and the equation $\calb^\Theta =\cal{D}$, the $\sigma$-algebra $\calb$ is generated by $\calb^\Theta$ and $\calb^\Gamma$.
The map $f^*$ is therefore surjective.
Assertion (ii) follows from \cite[Theorem 2.1]{ramsay}.
\end{proof}

We are now ready to prove stability of $\calg$.
We set $\cal{N}_1=\{ \, (Tx, x)\mid T\in \Theta,\ x\in X\, \}$.
This is a normal subgroupoid of $\calg$ by conditions (\ref{h}) and (\ref{i}).
Let $\cal{Q}_1$ be the quotient $\calg /\cal{N}_1$, which is a discrete measured groupoid on $(X_1, \mu_1)$.
We have the natural homomorphism $F_1\colon \calg \to \cal{Q}_1$ such that the induced map from $X$ into $X_1$ is equal to $f_1$.
The homomorphism $F_1$ is {\it class-surjective}, that is, for a.e.\ $h\in \cal{Q}_1$ and a.e.\ $x\in X$ such that $f_1(x)$ is equal to the source of $h$, there exists $g\in \calg$ such that $F_1(g)=h$ and the source of $g$ is equal to $x$.
Similarly, we set $\cal{N}_2=\{ \, (\gamma x, x)\mid \gamma \in \Gamma,\ x\in X\, \}$.
This is a normal subgroupoid of $\calg$, and let $\cal{Q}_2$ be the quotient $\calg /\cal{N}_2$.
Let $F_2\colon \calg \to \cal{Q}_2$ be the natural homomorphism.
We refer to \cite[Sections 3.4 and 3.5]{kida-bs} for normal subgroupoids and quotients by them.

Let $F\colon \calg \to \cal{Q}_1\times \cal{Q}_2$ be the homomorphism defined by $F(g)=(F_1(g), F_2(g))$ for $g\in \calg$.
We prove that $F$ is an isomorphism.
The restriction $F_1\colon \cal{N}_2\to \cal{Q}_1$ is class-surjective by conditions (\ref{h}) and (\ref{i}).
Similarly, the restriction $F_2\colon \cal{N}_1\to \cal{Q}_2$ is class-surjective.
It turns out from Lemma \ref{lem-b} (ii) that $F$ is surjective.
Combining conditions (\ref{a})--(\ref{c}), (\ref{h}) and (\ref{i}) with the equation $\calb^\Theta =\cal{D}$, we see that $\cal{Q}_1$ is isomorphic to $\calr_0$, the ergodic hyperfinite equivalence relation of type ${\rm II}_1$, and that the kernel of the restriction $F_1\colon \cal{N}_2\to \cal{Q}_1$ is the trivial groupoid on $(X, \mu)$.
It follows that $\cal{N}_1\cap \cal{N}_2$ is the trivial groupoid on $(X, \mu)$, and thus $F$ is injective.

Through an isomorphism between $\calr_0$ and $\calr_0\times \calr_0$, we obtain an isomorphism between $\calg$ and $\calg \times \calr_0$.
The proof of Theorem \ref{thm-js} is completed.
\end{proof}


\subsection{Stability of groups}

Recall that a discrete group is called stable if it has an ergodic, free, p.m.p.\ and stable action.
We prove Theorem \ref{thm-ac}, a useful criterion to get stability of groups, as a consequence of Lemmas \ref{lem-free} and \ref{lem-erg} below.

\begin{lem}\label{lem-free}
Any discrete group having an ergodic, p.m.p.\ and stable action is stable.
\end{lem}

\begin{proof}
Let $G$ be a discrete group and $G\c (X, \mu)$ an ergodic, p.m.p.\ and stable action.
We construct an ergodic, free, p.m.p.\ and stable action of $G$.
Set $\calg =G\ltimes (X, \mu)$.
Let $\Z \c (X_0, \mu_0)$ be an ergodic, free and p.m.p.\ action of the infinite cyclic group $\Z$ on a standard probability space.
Set $\calg_0=\Z \ltimes (X_0, \mu_0)$.
Since the action $G\c (X, \mu)$ is stable, we have an isomorphism $f\colon \calg \to \calg \times \calg_0$.
Set $H=G\times \Z$.
Let $\sigma \colon H=G\times \Z \c (X\times X_0, \mu \times \mu_0)$ be the coordinatewise action, which associates $\calg \times \calg_0$.
Let $\alpha \colon \calg \to H$ be the composition of $f$ with the projection from $\calg \times \calg_0$ onto $H$.

We define a measure space $(\Sigma, m)$ and a measure-preserving action of $G\times H$ on it as follows:
Set $\Sigma =X\times H$, and define a measure $m$ on $\Sigma$ as the product of $\mu$ and the counting measure on $H$.
Define an action of $G\times H$ on $\Sigma$ by
\[(g, h)(x, h')=(gx, \alpha(g, x)h'h^{-1})\quad \textrm{for}\ g\in G,\ h, h'\in H\ \textrm{and}\ x\in X.\]
This action $G\times H\c (\Sigma, m)$ defines a measure-equivalence coupling of $G$ and $H$.
We refer to \cite[Section 2.3]{kida-survey} for a coupling and its relationship to an isomorphism of groupoids.

Pick a free, mixing and p.m.p.\ action $G\c (Y, \nu)$ (e.g., the Bernoulli shift).
Let $\tau \colon H\c (Y, \nu)$ be the action obtained through the projection of $H=G\times \Z$ onto $G$.
Let $G\times H$ act on $(\Sigma \times Y, m\times \nu)$ by the formula
\[(g, h)(z, y)=((g, h)z, hy)\quad \textrm{for}\ g\in G,\ h\in H,\ z\in \Sigma\ \textrm{and}\ y\in Y.\]
This action defines a coupling of $G$ and $H$.
Let $e_H$ denote the neutral element of $H$.
The action of $H$ on the quotient $(\Sigma \times Y)/(G\times \{ e_H\})$ is isomorphic to the product action of $\sigma$ and $\tau$, and is therefore ergodic, free, p.m.p.\ and stable.
Since the action of $G\times H$ on $(\Sigma \times Y, m\times \nu)$ defines a coupling, we have an isomorphism between the groupoids
\[G\ltimes ((\Sigma \times Y)/(\{ e_G\} \times H))\quad \textrm{and}\quad H\ltimes ((\Sigma \times Y)/(G\times \{ e_H\})),\]
where $e_G$ is the neutral element of $G$.
We thus obtain an ergodic, free, p.m.p.\ and stable action of $G$.
\end{proof}

\begin{lem}\label{lem-erg}
Let $G$ be a discrete group and $\sigma \colon G\c (X, \mu)$ a p.m.p.\ action.
We set $\calg =G\ltimes (X, \mu)$, and suppose that $\sigma$ admits a non-trivial a.c.\ sequence in $[\calg]$.
Let $\theta \colon (X, \mu)\to (Z, \xi)$ be the ergodic decomposition for $\sigma$.
Let $\mu =\int_Z\mu_z \, d\xi(z)$ be the disintegration of $\mu$ with respect to $\theta$.
For a.e.\ $z\in Z$, we have the ergodic p.m.p.\ action $\sigma_z\colon G\c (X, \mu_z)$, and set $\calg_z=G\ltimes (X, \mu_z)$.

Then there exists a measurable subset $W$ of $Z$ of positive measure such that for a.e.\ $z\in W$, the action $\sigma_z$ admits a non-trivial a.c.\ sequence in $[\calg_z]$.
\end{lem}

\begin{proof}
Let $\calb$ be the algebra of measurable subsets of $X$.
By assumption, we have an a.c.\ sequence $(U_n)_{n=1}^\infty$ in $[\calg]$ for $\sigma$ and an a.i.\ sequence $(A_n)_{n=1}^\infty$ for $\sigma$ such that the measure $\mu(U_n^0A_n\bigtriangleup A_n)$ does not converge to 0 as $n\to \infty$.
Taking a subsequence, we may assume that there exists a real number $c>0$ with $\mu(U_n^0A_n\bigtriangleup A_n)\to c$ as $n\to \infty$.
We first note that the equation,
\begin{equation}\label{zz}
\mu(V^0A\bigtriangleup A)=\int_Z\mu_z(V^0A\bigtriangleup A)\, d\xi(z)\quad \textrm{for\ any}\ V\in [\calg]\ \textrm{and\ any}\ A\in \calb,
\end{equation}
holds.
For any $A\in \calb$, we have $\mu(U_n^0A\bigtriangleup A)\to 0$ as $n\to \infty$ because $(U_n)_n$ is a.c.\ for $\sigma$.
It thus follows from equation (\ref{zz}) that after taking a subsequence of $(U_n)_n$, for a.e.\ $z\in Z$ and any $A\in \calb$, we have $\mu_z(U_n^0A\bigtriangleup A)\to 0$ as $n\to \infty$. 
Similarly, it is shown that taking a subsequence of $(U_n)_n$ again, for a.e.\ $z\in Z$ and any $V\in [\calg]$, we have
\[\mu_z(\{ \, x\in X\mid (U_n\cdot V)x\neq (V\cdot U_n)x\, \})\to 0\]
as $n\to \infty$.
By equation (\ref{zz}), there exists a subsequence $(A_{n_k})_k$ of $(A_n)_n$ such that for a.e.\ $z\in Z$, the sequence $(A_{n_k})_k$ is a.i.\ for $\sigma_z$.
If for a.e.\ $z\in Z$, we had $\mu_z(U_{n_k}^0A_{n_k}\bigtriangleup A_{n_k})\to 0$ as $k\to \infty$, then equation (\ref{zz}) would imply $\mu(U_{n_k}^0A_{n_k}\bigtriangleup A_{n_k})\to 0$ as $k\to \infty$.
This is a contradiction.
The lemma follows.
\end{proof}

Combining Theorem \ref{thm-js}, Lemma \ref{lem-free} and Lemma \ref{lem-erg}, we obtain Theorem \ref{thm-ac}.



\section{Stable actions of central extensions}\label{sec-ce}

We prove Theorem \ref{thm-stab} (i) asserting that any discrete group $G$ having a central subgroup $C$ such that the pair $(G, C)$ does not have property (T) is stable.
Recall that for any irreducible unitary representation $(\pi, \calh)$ of a discrete group $G$ and for any central element $c$ of $G$, the unitary $\pi(c)$ is a scalar multiple of the identity operator on $\calh$.
This follows from Schur's lemma (\cite[Theorem A.2.2]{bhv}).
We thus regard $\pi(c)$ as an element of the torus $\mathbb{T} =\{ \, z\in \mathbb{C} \mid |z|=1\, \}$ if there is no confusion.

\begin{lem}\label{lem-pi-c}
Let $G$ be a discrete group with a central subgroup $C$.
Suppose that the pair $(G, C)$ does not have property (T).
Then there exist a sequence $(\pi_k)_{k=1}^\infty$ in the unitary dual $\widehat{G}$ of $G$ and a sequence $(c_k)_{k=1}^\infty$ in $C$ such that $\pi_k\to 1_G$ in $\widehat{G}$ as $k\to \infty$; and for any $k\in \N$, we have $|\pi_k(c_k)-1|>1$ and $|\pi_k(c_l)-1|\to 0$ as $l\to \infty$, where the symbol $|\cdot |$ denotes the absolute value of a complex number.
\end{lem}

\begin{proof}
The pair $(G, C)$ does not have property (T).
We thus have a sequence $(\tau_n)_{n=1}^\infty$ in $\widehat{G}$ such that $\tau_n\to 1_G$ as $n\to \infty$ and for any $n\in \N$, the representation $\tau_n$ has no non-zero $C$-invariant vector.

We inductively find an increasing sequence $(n_k)_{k=1}^\infty$ in $\N$ and a sequence $(c_k)_{k=1}^\infty$ in $C$ satisfying the following conditions (a) and (b):
\begin{enumerate}
\item[(a)] For any $k\in \N$, we have $|\tau_{n_k}(c_k)-1|>1$.
\item[(b)] For any $k, l\in \N$ with $k<l$, we have $|\tau_{n_k}(c_l)-1|<1/l$.
\end{enumerate}
We set $n_1=1$, and pick $c_1\in C$ with $|\tau_1(c_1)-1|>1$.
Such $c_1$ exists because the group $\tau_1(C)$ is non-trivial.
Suppose that we have an increasing sequence $n_1, n_2,\ldots, n_k$ in $\N$ and a sequence $c_1, c_2,\ldots, c_k$ in $C$ satisfying conditions (a) and (b) for them.
Pick a finite subset $F$ of $C$ such that the subset $(\tau_{n_1}\times \cdots \times \tau_{n_k})(C)$ of $\mathbb{T}^k$ is contained in the set
\[\bigcup_{d\in F}\{ \, (z_1,\ldots, z_k)\in \mathbb{T}^k\mid |z_i-\tau_{n_i}(d)|<1/(k+1)\ \textrm{for\ any}\ i=1,\ldots, k\, \}.\]
By the convergence $\tau_n\to 1_G$ as $n\to \infty$, there exists $n_{k+1}\in \N$ such that $n_{k+1}>n_k$ and for any $d\in F$, we have $|\tau_{n_{k+1}}(d)-1|<\sqrt{2}-1$.
The group $\tau_{n_{k+1}}(C)$ is non-trivial, and there thus exists $d_{k+1}\in C$ with $|\tau_{n_{k+1}}(d_{k+1})-1|>\sqrt{2}$.
By the choice of $F$, there exists $d\in F$ such that for any $i=1,\ldots, k$, we have $|\tau_{n_i}(d_{k+1})-\tau_{n_i}(d)|<1/(k+1)$.
We set $c_{k+1}=d^{-1}d_{k+1}$.
For any $i=1,\ldots, k$, we then have $|\tau_{n_i}(c_{k+1})-1|<1/(k+1)$ and
\begin{align*}
|\tau_{n_{k+1}}(c_{k+1})-1| & =|\tau_{n_{k+1}}(d_{k+1})-\tau_{n_{k+1}}(d)|\geq |\tau_{n_{k+1}}(d_{k+1})-1|-|\tau_{n_{k+1}}(d)-1|\\
&>\sqrt{2}-(\sqrt{2}-1)=1.
\end{align*}
The induction is completed.

Setting $\pi_k=\tau_{n_{k+1}}$ for $k\in \N$, we obtain desired sequences $(\pi_k)_k$ and $(c_k)_k$.
\end{proof}

\begin{proof}[Proof of Theorem \ref{thm-stab} (i)]
Let $G$ be a discrete group with a central subgroup $C$.
Suppose that the pair $(G, C)$ does not have property (T).
The aim is to show that $G$ is stable.
By Lemma \ref{lem-pi-c}, there exist a sequence $(\pi_n)_{n=1}^\infty$ in $\widehat{G}$ and a sequence $(c_n)_{n=1}^\infty$ in $C$ such that $\pi_n\to 1_G$ in $\widehat{G}$ as $n\to \infty$; for any $n\in \N$, we have $|\pi_n(c_n)-1|>1$; and for any $n\in \N$, we have $|\pi_n(c_m)-1|\to 0$ as $m\to \infty$.
For $n\in \N$, let $\calh_n$ denote the Hilbert space on which $\pi_n(G)$ acts.
The convergence $\pi_n\to 1_G$ as $n\to \infty$ implies that there exists a unit vector $\xi_n\in \calh_n$ for $n\in \N$ such that for any $g\in G$, we have $\langle \pi_n(g)\xi_n, \xi_n\rangle \to 1$ as $n\to \infty$.
For each $n\in \N$, applying Lemma \ref{lem-gauss} to $\pi_n$ and $\xi_n$, we obtain a p.m.p.\ action $G\c (\Omega_n, \nu_n)$ and a measurable subset $A_n$ of $\Omega_n$ with $\nu_n(A_n)=1/2$ satisfying the following conditions (1)--(3):
\begin{enumerate}
\item[(1)] For any $n\in \N$, we have $\nu_n(c_nA_n\bigtriangleup A_n)>1/3$.
\item[(2)] For any $g\in G$, we have $\nu_n(gA_n\bigtriangleup A_n)\to 0$ as $n\to \infty$.
\item[(3)] For any $n\in \N$ and any measurable subset $A$ of $\Omega_n$, we have $\nu_n(c_mA\bigtriangleup A)\to 0$ as $m\to \infty$.
\end{enumerate}
Condition (1) holds because for any $n\in \N$, the inequality ${\rm Re}\, \langle \pi_n(c_n)\xi_n, \xi_n\rangle ={\rm Re}\, \pi_n(c_n)<1/2$ holds.
We set $(\Omega, \nu)=\prod_{n=1}^\infty (\Omega_n, \nu_n)$, and define a p.m.p.\ action $\sigma \colon G\c (\Omega, \nu)$ as the diagonal action.
For $k\in \N$, we set $B_k=\{ \, (\omega_n)_n\in \Omega \mid \omega_k\in A_k\, \}$.
By condition (2), the sequence $(B_n)_{n=1}^\infty$ is a.i.\ for $\sigma$.
By condition (3), for any measurable subset $B$ of $\Omega$, we have $\nu(c_mB\bigtriangleup B)\to 0$ as $m\to \infty$.
By condition (1), the measure $\nu(c_nB_n\bigtriangleup B_n)$ does not converge to 0 as $n\to \infty$.
The sequence $(c_n)_{n=1}^\infty$ in $C$ therefore defines a non-trivial a.c.\ sequence for $\sigma$.
By Theorem \ref{thm-ac}, $G$ is stable.
\end{proof}



\section{Stable actions of central quotients}\label{sec-cq}

This section is devoted to the following:

\begin{proof}[Proof of Theorem \ref{thm-stab} (ii)]
We first fix the notation, and then give an outline of the proof.
Let $1\to C\to G\to \Gamma \to 1$ be an exact sequence of discrete groups such that $C$ is central in $G$ and the pair $(G, C)$ has property (T).
Suppose that there is an ergodic, free, p.m.p.\ and stable action $\sigma \colon G\c (X, \mu)$.
The aim is to show that $\Gamma$ also has such a stable action.
We set $\calg =G\ltimes (X, \mu)$.
Since $\sigma$ is stable, there exist a sequence $(U_n)_{n=1}^\infty$ in $[\calg]$ and an a.i.\ sequence $(A_n)_{n=1}^\infty$ for $\sigma$ satisfying the following conditions (1)--(3):
\begin{enumerate}
\item[(1)] For any measurable subset $A$ of $X$, we have $\mu(U_n^0A\bigtriangleup A)\to 0$ as $n\to \infty$.
\item[(2)] For any $V\in [\calg]$, we have $\mu(\{ \, x\in X\mid (U_n\cdot V)x\neq (V\cdot U_n)x\, \})\to 0$ as $n\to \infty$.
\item[(3)] For any $n\in \N$, the equation $U_n^0A_n=X\setminus A_n$ holds.
\end{enumerate}

We define $\cala =L^1(G\times X, \R)$ as the Banach space of real-valued integrable functions on $G\times X$, where $G\times X$ is equipped with the product measure of the counting measure on $G$ and $\mu$.
We also define a linear isometric representation $(\pi, \cala)$ of $G$ by
\[(\pi(g)f)(h, x)=f(g^{-1}hg, g^{-1}x)\quad \textrm{for}\ g, h\in G,\ f\in \cala \ \textrm{and}\ x\in X.\]
For a measurable subset $A$ of $G\times X$, we define $\chi_A\colon G\times X\to \{ 0, 1\}$ as the characteristic function on $A$.
For each $V\in [\calg]$, we have the unit vector $v\in \cala$ defined by the sum
\[v=\sum_{g\in G}\chi_{\{ g\} \times V^{-1}(g)}.\]
Let us call this vector $v$ the vector of $\cala$ {\it corresponding} to $V$.
We note that for any $V\in [\calg]$ with $v$ the vector of $\cala$ corresponding to $V$ and for any $g\in G$, the vector of $\cala$ corresponding to $g\cdot V\cdot g^{-1}\in [\calg]$ is $\pi(g)v$, where $g$ is identified with the constant map on $X$ with the value $g$.
For $n\in \N$, let $u_n$ denote the vector of $\cala$ corresponding to $U_n$.
Condition (2) implies that for any $g\in G$, we have $\Vert \pi(g)u_n-u_n\Vert \to 0$ as $n\to \infty$.
Namely, the sequence $(u_n)_n$ in $\cala$ is asymptotically $G$-invariant under $\pi$.

\medskip

\noindent {\bf An outline of the rest of the proof.}
Let $\cala^C$ denote the subspace of $\cala$ of $C$-invariant vectors.
The pair $(G, C)$ has property (T).
By Lemma \ref{lem-p}, we can find a vector $u_n'$ in $\cala^C$ close to $u_n$ for any large $n$.
Since $u_n$ corresponds to an element of $[\calg]$, the vector $u_n'$ can be approximated by a vector $v_n$ of $\cala^C$ corresponding to an element of $[\calg]$.
Let $V_n$ be the element of $[\calg]$ with $v_n$ corresponding to $V_n$.

Let $\theta \colon (X, \mu)\to (Z, \xi)$ be the ergodic decomposition for the action $C\c (X, \mu)$.
We have the natural action $\tau \colon \Gamma \c (Z, \xi)$.
Let $\rho \colon G\to \Gamma$ be the quotient homomorphism.
The map $V_n\colon X\to G$ is $C$-invariant, and thus induces the map $\bar{V}_n\colon Z\to \Gamma$ with $\bar{V}_n(\theta(x))=\rho(V_n(x))$ for a.e.\ $x\in X$.
The map $\bar{V}_n$ belongs to the full group of $\Gamma \ltimes (Z, \xi)$.
Since $V_n$ is close to $U_n$, the sequence $(V_n)_n$ is a non-trivial a.c.\ sequence for $\sigma$.
It follows that $(\bar{V}_n)_n$ is a non-trivial a.c.\ sequence for $\tau$.
By Theorem \ref{thm-ac}, we conclude that $\Gamma$ is stable.

\medskip

Following this outline, we give a precise proof.
As in Subsubsection \ref{subsubsec-lp}, let $P\colon \cala \to \cala$ be the linear map defined by
\[P(f)(g, x)=\mu_{\theta(x)}(f(g, \cdot))\quad \textrm{for}\ f\in \cala,\ g\in G\ \textrm{and}\ x\in X.\]
The map $P$ is the projection onto $\cala^C$.
For $n\in \N$, we set $u_n'=Pu_n$.
By Lemma \ref{lem-p}, taking a subsequence of $(u_n)_n$ if necessary, we may assume that for any $n\in \N$, the inequality
\begin{equation}\label{uprime}
\Vert u_n-u_n'\Vert <1/n
\end{equation}
holds.
We now approximate $u_n'$ by a vector in $\cala^C$ corresponding to an element of $[\calg]$.

Let $a$ be a real number with $0\leq a\leq 1$.
We define a function $E_a\colon [0, 1]\to \{ 0, 1\}$ by
\[E_a(t)=\begin{cases}
1 & \textrm{if}\ a\leq t\leq 1\\
0 & \textrm{if}\ 0\leq t<a.
\end{cases}\]
For a function $f\colon G\times X\to [0, 1]$, we define a function $f_a\colon G\times X\to \{ 0, 1\}$ by $f_a=E_a\circ f$.
The following argument using these functions is well known (e.g., see the proof of ``(ii) $\Rightarrow$ (iii)" in \cite[Theorem 5.14]{km}).

Fix $n\in \N$.
For a.e.\ $(g, x)\in G\times X$, we have $u_n(g, x)\in \{ 0, 1\}$ and $0\leq u_n'(g, x)\leq 1$.
For any number $a$ with $0<a\leq 1$, the equation $(u_n)_a=u_n$ holds, and the function $(u_n')_a$ is the characteristic function on the set $\{ u_n'\geq a\}$, and belongs to $\cala^C$ because $u_n'$ belongs to $\cala^C$.
We have the inequality
\begin{align*}
\int_0^1\Vert (u_n)_a-(u_n')_a\Vert \, da & =\int_0^1\left( \sum_{g\in G}\int_X|(u_n)_a(g, x)-(u_n')_a(g, x)|\, d\mu(x)\right)\, da\\
& =\sum_{g\in G}\int_X\left( \int_0^1|(u_n)_a(g, x)-(u_n')_a(g, x)|\, da \right)\, d\mu(x)\nonumber \\ 
& =\sum_{g\in G}\int_X |u_n(g, x)-u_n'(g, x)|\, d\mu(x)\nonumber \\
& =\Vert u_n-u_n'\Vert <1/n.\nonumber 
\end{align*}
There thus exists a number $a$ such that $0<a<1$ and $\Vert (u_n)_a-(u_n')_a\Vert <1/n$.
We fix such $a$, and set $u_n''=(u_n')_a$.
We have the inequality
\begin{equation}\label{uuprime2}
\Vert u_n-u_n''\Vert <1/n.
\end{equation}
The functions $u_n$ and $u_n''$ on $G\times X$ are regarded as maps from $X$ into $\ell^1(G)$, so that for a.e.\ $x\in X$, we have $u_n(x)=u_n(\cdot, x)$ and $u_n''(x)=u_n''(\cdot, x)$.
Since $u_n$ corresponds to an element of $[\calg]$, for a.e.\ $x\in X$, the function $u_n(x)$ is the Dirac function on an element of $G$.
The function $u_n''$ on $G\times X$ is a characteristic function on a measurable subset of $G\times X$.
For a.e.\ $x\in X$, the function $u_n''(x)$ is thus regarded as a sum of Dirac functions on finitely many, mutually distinct elements of $G$.
Namely, it is of the form $\delta_{g_1}+\delta_{g_2}+\cdots +\delta_{g_k}$ with $g_1, g_2,\ldots, g_k$ mutually distinct elements of $G$, where for $g\in G$, we denote by $\delta_g\in \ell^1(G)$ the Dirac function on $g$.
We set
\[E_1=\{ \, x\in X\mid  u_n''(x)\ {\rm is\ either\ 0\ or\ a\ sum\ of\ at\ least\ two\ Dirac\ functions.}\, \}.\]
We define a map $U_n''\colon X\setminus E_1\to G$ so that for any $x\in X$, $u_n''(x)$ is the Dirac function on $U_n''(x)$.
We also define a map $\varphi \colon X\setminus E_1\to X$ by $\varphi(x)=U_n''(x)x$ for $x\in X\setminus E_1$, and set
\[E_2=\{ \, x\in X\setminus E_1\mid |\varphi^{-1}(\varphi(x))|\geq 2\, \} \quad \textrm{and}\quad E=\{ \, x\in X\mid u_n(x)\neq u_n''(x)\, \}.\] 
The inclusion $E_1\subset E$ holds.
The map $u_n''\colon X\to \ell^1(G)$ is $C$-invariant, i.e., for any $c\in C$ and a.e.\ $x\in X$, we have $u_n''(cx)=u_n''(x)$.
The set $E_1$ is therefore $C$-invariant.
The map $\varphi$ is $C$-equivariant, i.e., for any $c\in C$ and a.e.\ $x\in X$, we have $\varphi(cx)=c\varphi(x)$.
The set $E_2$ is therefore $C$-invariant.

\begin{claim}
There exists a partition of $E_2$ into its $C$-invariant measurable subsets,
\begin{equation}\label{e2}
E_2=E_3\sqcup \left( \bigsqcup_{n\in N}F_n\right),
\end{equation}
with $N$ a countable set, such that $\varphi$ is injective on $E_3$ and on $F_n$ for any $n\in N$, and the equation $\varphi(E_3)=\varphi(E_2)$ holds up to null sets.
\end{claim}

\begin{proof}
Recall that we have the ergodic decomposition for the action $C\c (X, \mu)$, denoted by $\theta \colon (X, \mu)\to (Z, \xi)$.
We have the quotient homomorphism $\rho \colon G\to \Gamma$ and the ergodic p.m.p.\ action $\tau \colon \Gamma \c (Z, \xi)$.
We set $Y=X\setminus E_1$.
Since $E_1$ is $C$-invariant, there exists a measurable subset $\bar{Y}$ of $Z$ with $\theta^{-1}(\bar{Y})=Y$.
The map $U_n''\colon Y\to G$ is $C$-invariant, and thus induces the map $\bar{U}_n''\colon \bar{Y}\to G$ with $\bar{U}_n''(\theta(x))=U_n''(x)$ for a.e.\ $x\in Y$.

We define a map $\psi \colon \bar{Y}\to Z$ by $\psi(z)=(\rho \circ \bar{U}_n''(z))z$ for $z\in \bar{Y}$.
For any $z\in Z$, the set $\psi^{-1}(z)$ is countable.
We have a partition of $\bar{Y}$ into its measurable subsets,
\[\bar{Y}=S\sqcup \left( \bigsqcup_{n\in N}R_n\right),\]
with $N$ a countable set, such that $\psi$ is injective on $S$ and on $R_n$ for any $n\in N$, and the equation $\psi(S)=\psi(\bar{Y})$ holds.
We set $E_3=\theta^{-1}(S)$, and set $F_n=\theta^{-1}(R_n)$ for $n\in N$.
We then have equation (\ref{e2}).

We show that $\varphi$ is injective on $E_3$.
Pick two elements $x$, $y$ of $E_3$ with $\varphi(x)=\varphi(y)$, i.e., $U_n''(x)x=U_n''(y)y$.
Applying $\theta$ to this equation, we obtain $\psi(\theta(x))=\psi(\theta(y))$.
Since $\psi$ is injective on $S$, we have $\theta(x)=\theta(y)$.
We also have $U_n''(x)=\bar{U}_n''(\theta(x))=\bar{U}_n''(\theta(y))=U_n''(y)$, and therefore $U_n''(x)x=U_n''(y)y=U_n''(x)y$.
The equation $x=y$ holds because the action $\sigma \colon G\c (X, \mu)$ is free.
It follows that $\varphi$ is injective on $E_3$.

For any $n\in N$, replacing $E_3$ and $S$ by $F_n$ and $R_n$, respectively, in the last paragraph, we can also show that $\varphi$ is injective on $F_n$.
For a.e.\ $x\in Y$, the equation $\theta(\varphi(x))=\psi(\theta(x))$ holds by the definition of $\psi$.
It follows that the inclusion $\varphi(E_3)\subset \theta^{-1}(\psi(S))$ holds, and that for any $n\in N$, the inclusion $\varphi(F_n)\subset \theta^{-1}(\psi(R_n))$ holds.
Comparing the measures, up to null sets, we have the equations $\varphi(E_3)=\theta^{-1}(\psi(S))$ and $\varphi(F_n)=\theta^{-1}(\psi(R_n))$ for any $n\in N$.
Up to null sets, for any $n\in N$, the inclusion $\varphi(F_n)=\theta^{-1}(\psi(R_n))\subset \theta^{-1}(\psi(S))=\varphi(E_3)$ therefore holds, and the equation $\varphi(E_3)=\varphi(E_2)$ follows.
\end{proof}

For any $x\in X$, the set $\varphi^{-1}(x)$ is countable.
The vector $u_n$ corresponds to an element of $[\calg]$.
It follows that for a.e.\ $x\in E_2$, at most one point of $\varphi^{-1}(\varphi(x))$ belongs to $X\setminus E$.
The inequality $\mu(E_2\setminus E_3)\leq \mu(E)$ therefore holds.
For a.e.\ $x\in E$, we have $\Vert u_n(x)-u_n''(x)\Vert \geq 1$, where the norm is that in $\ell^1(G)$.
By inequality (\ref{uuprime2}), we have
\begin{equation}\label{ineq-e}
\mu(E)\leq \Vert u_n-u_n''\Vert <1/n\ {\rm and\ thus}\ \mu(E_2\setminus E_3)<1/n.
\end{equation}

We set $E_4=E_1\cup (E_2\setminus E_3)$.
Modifying the map $u_n''\colon X\to \ell^1(G)$ on $E_4$, we construct a vector in $\cala^C$ corresponding to an element of $[\calg]$.
The map $\varphi$ is injective on $X\setminus E_4$.
We set $F=\varphi(X\setminus E_4)$.
The equation $\mu(F)=\mu(X\setminus E_4)$ holds.
By ergodicity of $\sigma$, there exists a $C$-invariant measurable map $T\colon E_4\to G$ such that the map from $E_4$ into $X$ sending each $x\in E_4$ to $T(x)x$ is an isomorphism onto $X\setminus F$.
We define a map $V_n\colon X\to G$ by $V_n=U_n''$ on $X\setminus E_4$ and $V_n=T$ on $E_4$.
This $V_n$ is an element of $[\calg]$.
Let $v_n$ denote the vector of $\cala$ corresponding to $V_n$.
We have the inequality
\begin{align}\label{uprime2}
\Vert u_n''-v_n\Vert & =\int_{E_4}\Vert u_n''(x)-v_n(x)\Vert \, d\mu(x)\\
& \leq \int_{E_4}(\Vert u_n''(x)-u_n(x)\Vert +\Vert u_n(x)-v_n(x)\Vert)\, d\mu(x)\nonumber \\
& < (1/n)+2\mu(E_4)=(1/n)+2\mu(E_1)+2\mu(E_2\setminus E_3)<5/n.\nonumber
\end{align}
The map $V_n$ is $C$-invariant, and thus induces the map $\bar{V}_n\colon Z\to \Gamma$ with $\bar{V}_n(\theta(x))=\rho(V_n(x))$ for a.e.\ $x\in X$.
We set $\cal{Q}=\Gamma \ltimes (Z, \xi)$.
The map $\bar{V}_n$ belongs to $[\cal{Q}]$.

\begin{claim}\label{claim-q-ac}
The sequence $(\bar{V}_n)_{n=1}^\infty$ in $[\cal{Q}]$ is a.c.\ for $\tau$.
\end{claim}

\begin{proof}
By inequalities (\ref{ineq-e}) and (\ref{uprime2}), for any $n\in \N$, the inequality
\begin{align}\label{16}
\mu(\{ \, x\in X \mid U_nx\neq V_nx\, \})=\mu(\{ \, x\in X\mid u_n(x)\neq v_n(x)\, \})=\frac{\, \Vert u_n-v_n\Vert \, }{2}\\
\leq \frac{\, \Vert u_n-u_n''\Vert +\Vert u_n''-v_n\Vert \, }{2}< \frac{1}{\, 2\, }\left( \frac{1}{\, n\, }+\frac{5}{\, n\, }\right)=\frac{3}{\, n\, }\nonumber
\end{align}
holds.
For any $n\in \N$ and any measurable subset $A$ of $Z$, putting $B=\theta^{-1}(A)$, we have
\[\xi(\bar{V}_n^0A\bigtriangleup A)  =\mu(V_n^0B\bigtriangleup B)=2\mu(V_n^0B\setminus B) < 2\mu(U_n^0B\setminus B)+(6/n).\]
Condition (1) implies $\xi(\bar{V}_n^0A\bigtriangleup A)\to 0$ as $n\to \infty$.

Pick $\gamma \in \Gamma$ and $g\in G$ with $\rho(g)=\gamma$.
We identify $\gamma$ with the element of $[\cal{Q}]$ defined as the constant map on $Z$ with the value $\gamma$.
Similarly, we identify $g$ with the element of $[\calg]$ defined as the constant map on $X$ with the value $g$.
For any $n\in \N$, the inequality
\begin{align*}
\xi(\{ \, z\in Z\mid (\gamma \cdot \bar{V}_n)z\neq (\bar{V}_n\cdot \gamma)z\, \})\leq \mu(\{ \, x\in X\mid (g\cdot V_n)x\neq (V_n\cdot g)x\, \})\\
< \mu(\{ \, x\in X\mid (g\cdot U_n)x\neq (U_n\cdot g)x\, \})+(6/n)
\end{align*}
holds, where the last inequality holds because for any $x\in X$ with $(g\cdot V_n)x\neq (V_n\cdot g)x$, at least one of the three conditions, $(g\cdot U_n)x\neq (U_n\cdot g)x$, $U_nx\neq V_nx$ or $(U_n\cdot g)x\neq (V_n\cdot g)x$, holds.
Condition (2) implies $\xi(\{ \, z\in Z\mid (\gamma \cdot \bar{V}_n)z\neq (\bar{V}_n\cdot \gamma)z\, \})\to 0$ as $n\to \infty$.
\end{proof}

Finally, we construct an a.i.\ sequence for $\tau$.
We define $\calbb =L^1(X, \R)$ as the Banach space of real-valued integrable functions on $(X, \mu)$.
For a measurable subset $A$ of $X$, we denote by $\chi_A\colon X\to \{ 0, 1\}$ the characteristic function on $A$.
Recall that we have the a.i.\ sequence $(A_n)_{n=1}^\infty$ for $\sigma$.
The sequence $(\chi_{A_n})_{n=1}^\infty$ in $\calbb$ is asymptotically $G$-invariant under the Koopman representation of $G$ on $\calbb$.
As in Subsubsection \ref{subsubsec-lp}, we define a linear map $Q\colon \calbb \to \calbb$ by $Q(f)(x)=\mu_{\theta(x)}(f)$ for $f\in \calbb$ and $x\in X$.
The map $Q$ is the projection onto $\calbb^C$, the subspace of $\calbb$ of $C$-invariant functions on $X$.
For $n\in \N$, we set $w_n=Q\chi_{A_n}$.
By Lemma \ref{lem-q}, taking a subsequence of $(A_n)_n$ if necessary, we may assume that for any $n\in \N$, we have the inequality
\begin{equation}\label{aw}
\Vert \chi_{A_n}-w_n\Vert <1/n^2.
\end{equation}

Fix $n\in \N$.
For a.e.\ $x\in X$, we have $0\leq w_n(x)\leq 1$.
We divide $X$ into the three subsets, $X=D_1\sqcup D_2\sqcup B_n$, by setting
\[D_1=\left\{ w_n<1/n \right\},\quad D_2=\left\{ 1/n\leq w_n\leq 1-(1/n) \right\} \quad \textrm{and}\quad B_n=\left\{ w_n>1-(1/n)\right\}.\]
On $D_2$, the inequality $|\chi_{A_n}-w_n|\geq 1/n$ holds, and thus $\mu(D_2)/n\leq \Vert \chi_{A_n}-w_n\Vert <1/n^2$.
The inequality
\begin{equation}\label{wb}
\Vert w_n-\chi_{B_n}\Vert =\int_{X\setminus D_2}|w_n-\chi_{B_n}|\, d\mu +\int_{D_2}|w_n|\, d\mu \leq (1/n)+\mu(D_2)<2/n
\end{equation}
holds.
Since $w_n$ is a $C$-invariant function on $X$, the set $B_n$ is $C$-invariant.
We define $\bar{B}_n$ as the measurable subset of $Z$ with $\theta^{-1}(\bar{B}_n)=B_n$.

\begin{claim}\label{claim-q-ai}
The following assertions hold:
\begin{enumerate}
\item The sequence $(B_n)_{n=1}^\infty$ is a.i.\ for $\sigma$.
The sequence $(\bar{B}_n)_{n=1}^\infty$ is therefore a.i.\ for $\tau$.
\item The measure $\xi(\bar{V}_n^0\bar{B}_n\bigtriangleup \bar{B}_n)$ does not converge to 0 as $n\to \infty$.
\end{enumerate}
\end{claim}

\begin{proof}
By inequalities (\ref{aw}) and (\ref{wb}), for any $n\in \N$, we have the inequality
\begin{equation}\label{ab}
\mu(A_n\bigtriangleup B_n)=\Vert \chi_{A_n}-\chi_{B_n}\Vert \leq \Vert \chi_{A_n}-w_n\Vert +\Vert w_n-\chi_{B_n}\Vert < (1/n^2)+(2/n).
\end{equation}
The sequence $(A_n)_n$ is a.i.\ for $\sigma$, and assertion (i) therefore follows.
By inequalities (\ref{16}) and (\ref{ab}), we have the inequality
\begin{align*}
\xi(\bar{V}_n^0\bar{B}_n\bigtriangleup \bar{B}_n) & =\mu(V_n^0B_n\bigtriangleup B_n)> \mu(U_n^0B_n\bigtriangleup B_n)-(6/n)\\
& > \mu(U_n^0A_n\bigtriangleup A_n)-(2/n^2)-(10/n)=1-(2/n^2)-(10/n),
\end{align*}
where the last equation holds by condition (3).
Assertion (ii) follows.
\end{proof}

It follows from Claims \ref{claim-q-ac} and \ref{claim-q-ai} that $(\bar{V}_n)_{n=1}^\infty$ is a non-trivial a.c.\ sequence for $\tau$.
By Theorem \ref{thm-ac}, $\Gamma$ is stable.
\end{proof}




\end{document}